\documentclass[a4paper,reqno,11pt]{amsart}
\usepackage{amssymb}
\usepackage{amstext}
\usepackage{amsmath}
\usepackage{amscd}
\usepackage{amsthm}
\usepackage{amsfonts}
\usepackage{eepic}
\usepackage{enumerate}
\usepackage{graphicx}
\usepackage{latexsym}
\usepackage{mathrsfs}
\usepackage{tikz}
\usepackage[all]{xy}
\input xy
\xyoption{all}
\usepackage{pstricks}
\usepackage{lscape}
\usepackage{comment}

\numberwithin{equation}{section}
\newtheorem{theorem}{Theorem}[section]

\newtheorem{corollary}[theorem]{Corollary}
\newtheorem{lemma}[theorem]{Lemma}
\newtheorem{proposition}[theorem]{Proposition}

\theoremstyle{definition}
\newtheorem{definition}[theorem]{Definition}
\newtheorem{remark}[theorem]{Remark}
\newtheorem{example}[theorem]{Example}
\newcommand{\add}{\operatorname{\mathsf{add}}\nolimits}

\newcommand{\End}{\operatorname{End}\nolimits}
\newcommand{\Ext}{\operatorname{Ext}\nolimits}

\newcommand{\Hom}{\operatorname{Hom}\nolimits}
\renewcommand{\Im}{\operatorname{Im}\nolimits}

\newcommand{\kD}{\operatorname{D}}
\newcommand{\Ker}{\operatorname{Ker}\nolimits}
\newcommand{\Cok}{\operatorname{Cok}\nolimits}

\renewcommand{\mod}{\operatorname{\mathsf{mod}}\nolimits}

\newcommand{\thick}{\operatorname{\mathsf{thick}}\nolimits}
\newcommand{\Sub}{\operatorname{\mathsf{Sub}}\nolimits}
\newcommand{\gl}{\operatorname{gl.dim}\nolimits}

\newcommand{\supp}{\operatorname{Supp}\nolimits}
\newcommand{\Deg}{\operatorname{deg}\nolimits}

\newcommand{\injdim}{\operatorname{inj.dim}\nolimits}

\newcommand{\proj}{\operatorname{proj}\nolimits}
\newcommand{\Max}{\operatorname{max}\nolimits}
\newcommand{\xto}[1]{\xrightarrow{#1}}

\renewcommand{\top}{\operatorname{top}\nolimits}
\begin{document}
\title[Tilting theory of preprojective algebras and $c$-sortable elements]{Tilting theory of preprojective algebras and $c$-sortable elements}
\author[Y. Kimura]{Yuta Kimura}
\address{Graduate School of Mathematics, Nagoya University, Frocho, Chikusaku, Nagoya, 464-8602, Japan}
\email{m13025a@math.nagoya-u.ac.jp}
\date{\today}
\begin{abstract}
For a finite acyclic quiver $Q$ and the corresponding preprojective algebra $\Pi$, we study the factor algebra $\Pi_w$ associated with an element $w$ in the Coxeter group of $Q$ introduced by Buan-Iyama-Reiten-Scott.
The algebra $\Pi_w$ has a natural $\mathbb{Z}$-grading.
We prove that $\underline{\Sub}^{\mathbb{Z}}\Pi_w$ has a tilting object $M$ if $w$ is $c$-sortable.
Moreover, we show that the endomorphism algebra of $M$ is isomorphic to the stable Auslander algebra of a certain torsion free class of $\mod kQ$.
\end{abstract}
\maketitle
\section{Introduction}
The preprojective algebra $\Pi$ of a finite acyclic quiver $Q$ plays important roles in representation theory of algebras.
One of them is categorifications of cluster algebras introduced by Fomin-Zelevinsky \cite{FZ}.
Namely, $\Pi$ gives $2$-Calabi-Yau triangulated categories ($2$-CY for short) with cluster tilting objects.

For a Dynkin quiver $Q$, Geiss-Leclerc-Schr\"{o}er showed that the stable category $\underline{\mod}\,\Pi$ of $\Pi$ is a $2$-CY triangulated category with cluster tilting objects \cite{GLS}.
More generally,  for a finite acyclic quiver $Q$ and an element $w$ of the Coxeter group of $Q$, Buan-Iyama-Reiten-Scott constructed an algebra $\Pi_{w}$.
They showed that $\Pi_{w}$ is an Iwanaga-Gorenstein algebra with injective dimension at most one, and therefore the category $\Sub\,\Pi_{w}$ of submodules of free $\Pi_{w}$-modules is a Frobenius category. 
Moreover they showed that the stable category $\underline{\Sub}\,\Pi_{w}$ is a $2$-CY triangulated category and contains a cluster tilting object associated with each reduced expression of $w$ \cite{BIRSc}.
When $Q$ is of Dynkin type and $w$ is the longest element of the Coxeter group, one obtains the stable category $\underline{\mod}\,\Pi$.

The cluster category ${\mathsf C}_Q$ of finite acyclic quiver $Q$ is another example of a $2$-CY triangulated category with cluster tilting objects \cite{BMRRT}.
More generally, for a finite dimensional algebra $A$ of global dimension at most two, Amiot's  cluster category ${\mathsf C}_A$ is a $2$-CY triangulated category with a cluster tilting object $A$ if it is Hom-finite \cite{Amio}.
Amiot-Reiten-Todorov \cite{ART} showed that there exists a close connection between these $2$-CY categories $\underline{\Sub}\,\Pi_w$ and ${\mathsf C}_A$.
Namely, for any element $w$ of the Coxeter group of a finite acyclic quiver $Q$, 
there exists a triangle equivalence
\begin{align}\label{aireq}
\underline{\Sub}\,\Pi_{w}\simeq{\mathsf C}_{\Gamma_{w}}
\end{align}
for some finite dimensional algebra $\Gamma_w$.

The aim of this paper is to construct a derived category version of the result of Amiot-Reiten-Todorov.
More precisely, we regard $\Pi_w$ as a $\mathbb{Z}$-graded algebra whose grading is given by the orientation of our quiver $Q$, and consider the stable category $\underline{\Sub}^{\mathbb{Z}}\Pi_w$ of graded $\Pi_w$-submodules of graded free $\Pi_w$-modules.
We denote by $c$ the Coxeter element corresponding to the orientation of $Q$.
In this setting, the category $\underline{\Sub}^{\mathbb{Z}}\Pi_w$ behaves nicely if $w$ is a $c$-sortable element (see Section \ref{preliminaries} for Definition), and we have the following Theorem.

\begin{theorem}[see Theorem \ref{thm}]\label{intro1}
If $w=c^{(0)}\cdots c^{(m)}$ is $c$-sortable, then  $\underline{\Sub}^{\mathbb{Z}}\Pi_w$ has a tilting object $M=\bigoplus _{i=0}^{m}(\Pi_{c^{(0)}\cdots c^{(i)}})(i)$.
\end{theorem}

Our second main result gives a simple description of $A_{w}:=\underline{\End}_{\Pi_w}^{\mathbb{Z}}(M)$.
Let $M_{0}$ be the degree zero part of $M$.
By \cite[Theorem 3.11]{AIRT}, there exists a tilting $kQ$-module $T$ such that $\Sub T$ has an additive generator $M_{0}$ (see Theorem \ref{airt} for details).
Using this notation, we have the following theorem.

\begin{theorem}[see Theorems \ref{usubz=dend}, \ref{endalg}, and \ref{end2}]\label{intro2}
We have the following:
\begin{itemize}
	\item[(a)] There exists an isomorphism of algebras $A_{w}\simeq \End_{kQ}(M_{0})/[T]$.
	\item[(b)] The global dimension of $A_{w}$ is at most two.
	\item[(c)] We have a triangle equivalence \[\underline{\Sub}^{\mathbb{Z}}\Pi_{w}\simeq{\mathsf D}^{{\rm b}}(A_{w}).\]
\end{itemize}
\end{theorem}

In the forthcoming paper, we will study the relationship between Amiot-Reiten-Todorov equivalence (\ref{aireq}) and our equivalence Theorem \ref{intro2} (c), that is, we have a commutative diagram of functors.

This paper is organized as follows.
In Section \ref{preliminaries}, we give some notations used in this paper and some results of \cite{BIRSc} which will be used in the proofs.
In Section \ref{gradedstrof}, we prove some basic properties of the grading of $\Pi_w$ when $w$ is a $c$-sortable element and recall some results of \cite{AIRT}.
In Section \ref{tilting}, we prove Theorem \ref{intro1}.
In Section \ref{theendomorphism}, we prove (a) and (c) of Theorem \ref{intro2}.
Theorem \ref{intro2} (b) follows from a general result in Section \ref{gldimofend} on the global dimension of relative version of stable Auslander algebras.

In this paper, we denote by $k$ an algebraically closed field.
All algebras are $k$-algebras, and all graded algebras are $\mathbb{Z}$-graded $k$-algebras.
We always deal with finitely generated left modules.
For an algebra $A$, we denote by $\mod\,A$ the category of finitely generated $A$-modules.
For a graded algebra $A$, we denote by $\mod^{\mathbb{Z}}A$ the category of finitely generated $\mathbb{Z}$-graded $A$-modules with degree zero morphisms.
For a category $\mathcal{C}=\mod A$ or $\mod^{\mathbb{Z}}A$ and $M\in \mathcal{C}$, we denote by $\add\,M$ the additive closure of $M$ in $\mathcal{C}$.
The composition of morphisms $f: X \to Y$ and $g: Y \to Z$ is denoted by $fg=g\circ f: X \to Z$.
For two arrows $\alpha$, $\beta$ of a quiver such that the target point of $\alpha$ is the start point of $\beta$, we denote by $\alpha\beta$ the composition of $\alpha$ and $\beta$.
\section{Preliminaries}\label{preliminaries}
We fix a finite acyclic quiver $Q=(Q_0,Q_1,s,t)$, where $Q_0=\{1,\ldots,n\}$ is the set of vertices, $Q_1$ is the set of arrows, 
and an arrow $\alpha$ goes from $s(\alpha)$ to $t(\alpha)$.
Let $kQ$ be the path algebra of $Q$ over $k$, and 
for a vertex $u$ of $Q$, we denote by $e_{u}$ the corresponding idempotent of $kQ$.
The {\it double quiver} $\overline{Q}=(\overline{Q}_0,\overline{Q}_1,s,t)$ of a quiver $Q$ is defined by $\overline{Q}_0=Q_0$, $\overline{Q}_1=Q_1\sqcup \{ \alpha^{\ast}: t(\alpha)\to s(\alpha) \mid \alpha \in Q_1 \}$.
Then we define the {\it preprojective algebra} $\Pi$ of $Q$ by
	\begin{align}
	\Pi :=k\overline{Q}/\langle \displaystyle\sum\limits_{\alpha\in Q_1} \alpha\alpha^{\ast}-\alpha^{\ast}\alpha \rangle . \notag
	\end{align}
The {\it Coxeter group} $W=W_{Q}$ of $Q$ is the group generated by the set  $\{s_{u} \mid u\in Q_0\}$ with relations $s_{u}^2=1$, $s_{u}s_{v}=s_{v}s_{u}$ if there exist no arrows between $u$ and $v$, and $s_us_vs_u=s_vs_us_v$ if there exists exactly one arrow between $u$ and $v$.
An expression $w=s_{u_1}s_{u_2}\cdots s_{u_l}$ is {\it reduced} if for any  other expression $w=s_{v_1}s_{v_2}\cdots s_{v_m}$, we have $l\leq m$.
For an element $w$ in $W$ with a reduced expression $w=s_{u_1}s_{u_2}\cdots s_{u_l}$, let $\supp(w):=\{ u_{1}, u_{2}, \ldots, u_{l} \}\subset Q_{0}$, which is independent of the choice of a reduced expression of $w$ (see \cite[Corollary 1.4.8 (ii)]{BB}).
\begin{definition}
An element $c\in W$ is called a {\it Coxeter element} if there is an expression 
$c=s_{u_1}s_{u_2}\cdots s_{u_n}$, where $u_1,\ldots,u_n$ is a permutation of $1,\ldots,n$.
In this paper, we only consider a Coxeter element $c$ satisfying $e_{u_j}(kQ)e_{u_i}=0$ for $i < j$ which is uniquely determined by the orientation of $Q$.
\end{definition}
\begin{definition}\cite{R}\label{csortable}
Let $c$ be a Coxeter element of $W$. An element $w\in W$ is called a {\it $c$-sortable element} if there exists a reduced expression $w=s_{u_1}s_{u_2}\cdots s_{u_l}$ of the form $s_{u_1}s_{u_2}\cdots s_{u_l}=c^{(0)}c^{(1)}\cdots c^{(m)}$, where each $c^{(i)}$ is subsequence of $c$ and
	\begin{align}
	\supp(c^{(m)}) \subset \supp(c^{(m-1)}) \subset \cdots \subset \supp(c^{(0)}) \subset Q_0. \notag
	\end{align}
\end{definition}
Let $u$ be a vertex of $Q$.
We define the  two-sided ideal $I_u$ of $\Pi$ by \[I_u:=\Pi(1-e_u)\Pi.\]
Let $w=s_{u_1}s_{u_2}\cdots s_{u_l}$ be a reduced expression of $w\in W$. We define a two-sided ideal $I_w$ of $\Pi$ by \[I_w:=I_{u_1}I_{u_2}\cdots I_{u_l}.\]
Note that $I_w$ is independent of the choice of a reduced expression of $w$ by  \cite[Theorem I\hspace{-.1em}I\hspace{-.1em}I. 1.9]{BIRSc}.
We define the algebra $\Pi_w$ by \[\Pi_w:=\Pi/I_w.\]
We denote by $\Sub\Pi_w$ the full subcategory of $\mod\Pi_w$ of submodules of finitely generated free $\Pi_w$-modules.

We recall the following properties:
\begin{proposition} \cite{BIRSc}\label{birs}
For an element $w$ of the Coxeter group, we have the following results.
\begin{itemize}
\item[(a)] If $Q$ is non-Dynkin, then there exists an isomorphism of algebras $\Pi \xto{\sim} \End_{\Pi}(I_{w})$, where this isomorphism is given by $x\mapsto (\cdot x)$.
\item[(b)] The algebra $\Pi_w$ is finite dimensional and {\emph Iwanaga-Gorenstein} of dimension at most one, that is, $\injdim {}_{\Pi_w}(\Pi_w)\leq 1$ and $\injdim (\Pi_w)_{\Pi_w}\leq 1$.
\item[(c)] The category $\Sub\Pi_w$ is a Frobenius category.
\item[(d)] The stable category $\underline{\Sub}\,\Pi_w$ of $\Sub\Pi_w$ is $2$-Calabi-Yau triangulated category, that is, for any objects $X,Y\in\underline{\Sub}\,\Pi_w$ there is a functorial isomorphism $\underline{\Hom}_{\Pi_w}(X,Y)\simeq\kD\underline{\Hom}_{\Pi_w}(Y,X[2])$, where $\kD=\Hom_k(-,k)$. 
\item[(e)] For any reduced expression $w=s_{u_1}s_{u_2}\cdots s_{u_l}$, the object $T=\bigoplus_{i=1}^l\Pi_{u_1u_2\cdots u_i}$ is in $\Sub\,\Pi_{w}$.
\item[(f)]
For any reduced expression $w=s_{u_1}s_{u_2}\cdots s_{u_l}$, the object $T$ of (e) is a cluster tilting object of $\Sub\,\Pi_w$, that is, $\add T=\{X\in\mod \Pi_{w}\mid \Ext_{\Pi_{w}}^{1}(X,T)=0\}$.
\end{itemize}
\end{proposition}
Next we introduce the grading of a preprojective algebra.
We regard the path algebra $k\overline{Q}$ as a graded algebra by the following grading:
	\begin{align}
	\Deg \beta = \begin{cases}
				1 & \text{$\beta = \alpha^{\ast},\alpha \in Q_1$} \\
				0 & \text{$\beta = \alpha , \alpha \in Q_1.$}
			\end{cases}\notag
	\end{align}
Since the element $\sum\limits_{\alpha\in Q_1}(\alpha\alpha^{\ast}-\alpha^{\ast}\alpha)$ in $k\overline{Q}$ is homogeneous of degree $1$, the grading of $k\overline{Q}$ naturally gives a grading on the preprojective algebra $\Pi=\bigoplus\limits_{i\geq 0}\Pi_i$.
A $\mathbb{Z}$-algebra $A$ is said to be {\it positively graded} if $A_{i}=0$ for any $i< 0$.
Preprojective algebras are positively graded with respect to the above grading.
\begin{remark}
\begin{itemize}
\item[(a)] We have $\Pi_0=kQ$, since $\Pi_0$ is spanned by all paths of degree $0$.
\item[(b)] For any $w\in W$, the ideal $I_w$ of $\Pi$ is a homogeneous ideal of $\Pi$ since so is each $I_u$.
\item[(c)] In particular, the factor algebra $\Pi_w$ is a graded algebra.
\end{itemize}
\end{remark}
Let $X=\bigoplus_{i\in\mathbb{Z}}X_i$ be a graded module  over a positively graded algebra.
For any integer $j$, we define the shifted graded module $X(j)$ by $(X(j))_i=X_{i+j}$.
Moreover, for any integer $j$, we define a graded submodule $X_{\geq j}$ of $X$ by
	\begin{align}
	(X_{\geq j})_i=\begin{cases}
					X_i & \text{$i\geq j$} \\
					0 & \text{else}
					\end{cases} \notag
	\end{align}
and define a graded factor module $X_{\leq j}$ of $X$ by $X_{\leq j}=X/(X_{\geq {j+1}})$.
For $i,j\in\mathbb{Z}$, let $X_{[i,j]}=(X_{\leq j})_{\geq i}$.

We denote by $\proj^{\mathbb{Z}}\Pi_{w}$ the full subcategory of $\mod^{\mathbb{Z}}\Pi_{w}$ of graded projective $\Pi_{w}$-modules.
Let $\Sub^{\mathbb{Z}}\Pi_w$ be the full subcategory of $\mod^{\mathbb{Z}}\Pi_w$ of submodules of graded free $\Pi_w$-modules, that is,
 	\begin{align}
	\Sub^{\mathbb{Z}}\Pi_w=\biggl\{ X\in\mod^{\mathbb{Z}}\Pi_w\mid X \hspace{0.2cm} \textnormal{is a submodule of} \hspace{0.2cm} \bigoplus\limits_{i=1}^m\Pi_w(j_i),\,\, m,j_i\in\mathbb{Z}, m\geq 0 \biggr\}. \notag
	\end{align}
We have the degree forgetful functor $\rho : \mod^{\mathbb{Z}}\Pi_{w}\to\mod\Pi_{w}$.
Then we have the following equalities.
\begin{align}
\Sub^{\mathbb{Z}}\Pi_{w}&=\biggl\{ X\in\mod^{\mathbb{Z}}\Pi_{w} \mid \rho(X)\in\Sub\,\Pi_{w} \biggl\},\label{subz=sub}\\
&=\biggl\{ X\in\mod^{\mathbb{Z}}\Pi_{w} \mid \Ext_{\Pi_{w}}^{>0}(\rho(X), \Pi_{w})=0,  \forall i>0 \biggl\},\notag\\
&=\biggl\{ X\in\mod^{\mathbb{Z}}\Pi_{w} \mid \Ext_{\mod^{\mathbb{Z}}\Pi_{w}}^{>0}(X, \Pi_{w}(i))=0,  \forall i\in\mathbb{Z} \biggl\}.\label{subz=ext}
\end{align}
Clearly $\Sub^{\mathbb{Z}}\Pi_{w}$ has enough projectives and is closed under direct summands.
By (\ref{subz=ext}), $\Sub^{\mathbb{Z}}\Pi_{w}$ is closed under extensions.
For any $X\in\Sub^{\mathbb{Z}}\Pi_{w}$, there exists an injective left ($\proj^{\mathbb{Z}}\Pi_{w}$)-approximations of $X$.
Thus $\Sub^{\mathbb{Z}}\Pi_{w}$ has enough injectives by (\ref{subz=ext}).
It is easy to see that the projective objects and the injective objects of $\Sub^{\mathbb{Z}}\Pi_{w}$ coincide and equals to $\proj^{\mathbb{Z}}\Pi_{w}$.
Therefore $\Sub^{\mathbb{Z}}\Pi_w$ is a Frobenius category.
We have a triangulated category $\underline{\Sub}^{\mathbb{Z}}\Pi_w$.
In this paper, we get a tilting object in this category.

We give one example which illustrates grading on the algebra $\Pi_{w}$ when $w$ is $c$-sortable.
\begin{example}
Let $Q$ be a quiver
	\begin{xy} (0,0)="O",
	"O"+<0cm,0.35cm>="11"*{1},
	"11"+<-0.8cm,-0.7cm>="21"*{2},
	"21"+/r1.6cm/="22"*{3},
	
	\ar"11"+/dl/;"21"+/u/
	\ar"21"+/r/;"22"+/l/
	\ar"11"+/dr/;"22"+/u/
	\end{xy}.
Then we have a graded algebra $\Pi=\Pi e_1\oplus\Pi e_2\oplus\Pi e_3$, and these are represented by their radical filtrations, which correspond to the horizontal layers of simples, as follows:
	$$
	\Pi e_{1}=\begin{xy} (0,0)="O",
	"O"+<0cm,1.2cm>="11"*{\bf{1}},
	"11"+<-0.3cm,-0.5cm>="21"*{2},
	"21"+<-0.3cm,-0.5cm>="31"*{3},
	"31"+<-0.3cm,-0.5cm>="41"*{1},
	"41"+<-0.3cm,-0.5cm>="51"*{2},
	"51"+<-0.3cm,-0.5cm>="61"*{3},
	"61"+<-0.3cm,-0.5cm>="71",
	"21"+/r0.6cm/="22"*{3},
	"31"+/r0.6cm/="32"*{1},
	"32"+/r0.6cm/="33"*{2},
	"41"+/r0.6cm/="42"*{2},
	"42"+/r0.6cm/="43"*{3},
	"43"+/r0.6cm/="44"*{1},
	"51"+/r0.6cm/="52"*{3},
	"52"+/r0.6cm/="53"*{1},
	"53"+/r0.6cm/="54"*{2},
	"54"+/r0.6cm/="55"*{3},
	"61"+/r0.6cm/="62"*{1},
	"62"+/r0.6cm/="63"*{2},
	"63"+/r0.6cm/="64"*{3},
	"64"+/r0.6cm/="65"*{1},
	"65"+/r0.6cm/="66"*{2},
	"71"+/r0.6cm/="72",
	"72"+/r0.6cm/="73",
	"73"+/r0.6cm/="74",
	"74"+/r0.6cm/="75",
	"75"+/r0.6cm/="76",
	"76"+/r0.6cm/="77",
	
	\ar@{-}"21"+/dr/;"32"+/u/<-2pt>
	\ar@{-}"22"+/dl/;"32"+/u/<2pt>
	\ar@{-}"22"+/dr/;"33"+/u/<-2pt>
	\ar@{-}"31"+/dl/;"41"+/u/<2pt>
	\ar@{-}"31"+/dr/;"42"+/u/<-2pt>
	\ar@{-}"33"+/dr/;"44"+/u/<-2pt>
	\ar@{-}"42"+/dr/;"53"+/u/<-2pt>
	\ar@{-}"43"+/dl/;"53"+/u/<2pt>
	\ar@{-}"43"+/dr/;"54"+/u/<-2pt>
	\ar@{-}"51"+/dr/;"62"+/u/<-2pt>
	\ar@{-}"52"+/dl/;"62"+/u/<2pt>
	\ar@{-}"52"+/dr/;"63"+/u/<-2pt>
	\ar@{-}"54"+/dr/;"65"+/u/<-2pt>
	\ar@{-}"55"+/dl/;"65"+/u/<2pt>
	\ar@{-}"55"+/dr/;"66"+/u/<-2pt>
	\ar@{-}"61"+/dl/;"71"+/u/<2pt>
	\ar@{-}"61"+/dr/;"72"+/u/<-2pt>
	\ar@{-}"63"+/dr/;"74"+/u/<-2pt>
	\ar@{-}"64"+/dl/;"74"+/u/<2pt>
	\ar@{-}"64"+/dr/;"75"+/u/<-2pt>
	\ar@{-}"66"+/dr/;"77"+/u/<-2pt>
	\end{xy},
\qquad
	\Pi e_{2}=\begin{xy} (0,0)="O",
	"O"+<0cm,1.2cm>="11"*{{\bf 2}},
	"11"+<-0.3cm,-0.5cm>="21"*{3},
	"21"+<-0.3cm,-0.5cm>="31"*{1},
	"31"+<-0.3cm,-0.5cm>="41"*{2},
	"41"+<-0.3cm,-0.5cm>="51"*{3},
	"51"+<-0.3cm,-0.5cm>="61"*{1},
	"61"+<-0.3cm,-0.5cm>="71",
	"21"+/r0.6cm/="22"*{{\bf 1}},
	"31"+/r0.6cm/="32"*{2},
	"32"+/r0.6cm/="33"*{3},
	"41"+/r0.6cm/="42"*{3},
	"42"+/r0.6cm/="43"*{1},
	"43"+/r0.6cm/="44"*{2},
	"51"+/r0.6cm/="52"*{1},
	"52"+/r0.6cm/="53"*{2},
	"53"+/r0.6cm/="54"*{3},
	"54"+/r0.6cm/="55"*{1},
	"61"+/r0.6cm/="62"*{2},
	"62"+/r0.6cm/="63"*{3},
	"63"+/r0.6cm/="64"*{1},
	"64"+/r0.6cm/="65"*{2},
	"65"+/r0.6cm/="66"*{3},
	"71"+/r0.6cm/="72",
	"72"+/r0.6cm/="73",
	"73"+/r0.6cm/="74",
	"74"+/r0.6cm/="75",
	"75"+/r0.6cm/="76",
	"76"+/r0.6cm/="77",
	
	\ar@{-}"11"+/dr/;"22"+/u/<-2pt>
	\ar@{-}"21"+/dl/;"31"+/u/<2pt>
	\ar@{-}"21"+/dr/;"32"+/u/<-2pt>
	\ar@{-}"32"+/dr/;"43"+/u/<-2pt>
	\ar@{-}"33"+/dl/;"43"+/u/<2pt>
	\ar@{-}"33"+/dr/;"44"+/u/<-2pt>
	\ar@{-}"41"+/dr/;"52"+/u/<-2pt>
	\ar@{-}"42"+/dl/;"52"+/u/<2pt>
	\ar@{-}"42"+/dr/;"53"+/u/<-2pt>
	\ar@{-}"44"+/dr/;"55"+/u/<-2pt>
	\ar@{-}"51"+/dl/;"61"+/u/<2pt>
	\ar@{-}"51"+/dr/;"62"+/u/<-2pt>
	\ar@{-}"53"+/dr/;"64"+/u/<-2pt>
	\ar@{-}"54"+/dl/;"64"+/u/<2pt>
	\ar@{-}"54"+/dr/;"65"+/u/<-2pt>
	\ar@{-}"62"+/dr/;"73"+/u/<-2pt>
	\ar@{-}"63"+/dl/;"73"+/u/<2pt>
	\ar@{-}"63"+/dr/;"74"+/u/<-2pt>
	\ar@{-}"65"+/dr/;"76"+/u/<-2pt>
	\ar@{-}"66"+/dl/;"76"+/u/<2pt>
	\ar@{-}"66"+/dr/;"77"+/u/<-2pt>
	\end{xy},
\qquad
	\Pi e_{3}=\begin{xy} (0,0)="O",
	"O"+<0cm,1.2cm>="11"*{{\bf 3}},
	"11"+<-0.3cm,-0.5cm>="21"*{{\bf 1}},
	"21"+<-0.3cm,-0.5cm>="31"*{2},
	"31"+<-0.3cm,-0.5cm>="41"*{3},
	"41"+<-0.3cm,-0.5cm>="51"*{1},
	"51"+<-0.3cm,-0.5cm>="61"*{2},
	"61"+<-0.3cm,-0.5cm>="71",
	"21"+/r0.6cm/="22"*{{\bf 2}},
	"31"+/r0.6cm/="32"*{3},
	"32"+/r0.6cm/="33"*{{\bf 1}},
	"41"+/r0.6cm/="42"*{1},
	"42"+/r0.6cm/="43"*{2},
	"43"+/r0.6cm/="44"*{3},
	"51"+/r0.6cm/="52"*{2},
	"52"+/r0.6cm/="53"*{3},
	"53"+/r0.6cm/="54"*{1},
	"54"+/r0.6cm/="55"*{2},
	"61"+/r0.6cm/="62"*{3},
	"62"+/r0.6cm/="63"*{1},
	"63"+/r0.6cm/="64"*{2},
	"64"+/r0.6cm/="65"*{3},
	"65"+/r0.6cm/="66"*{1},
	"71"+/r0.6cm/="72",
	"72"+/r0.6cm/="73",
	"73"+/r0.6cm/="74",
	"74"+/r0.6cm/="75",
	"75"+/r0.6cm/="76",
	"76"+/r0.6cm/="77",
	
	\ar@{-}"11"+/dl/;"21"+/u/<2pt>
	\ar@{-}"11"+/dr/;"22"+/u/<-2pt>
	\ar@{-}"22"+/dr/;"33"+/u/<-2pt>
	\ar@{-}"31"+/dr/;"42"+/u/<-2pt>
	\ar@{-}"32"+/dl/;"42"+/u/<2pt>
	\ar@{-}"32"+/dr/;"43"+/u/<-2pt>
	\ar@{-}"41"+/dl/;"51"+/u/<2pt>
	\ar@{-}"41"+/dr/;"52"+/u/<-2pt>
	\ar@{-}"43"+/dr/;"54"+/u/<-2pt>
	\ar@{-}"44"+/dl/;"54"+/u/<2pt>
	\ar@{-}"44"+/dr/;"55"+/u/<-2pt>
	\ar@{-}"52"+/dr/;"63"+/u/<-2pt>
	\ar@{-}"53"+/dl/;"63"+/u/<2pt>
	\ar@{-}"53"+/dr/;"64"+/u/<-2pt>
	\ar@{-}"55"+/dr/;"66"+/u/<-2pt>
	\ar@{-}"61"+/dr/;"72"+/u/<-2pt>
	\ar@{-}"62"+/dl/;"72"+/u/<2pt>
	\ar@{-}"62"+/dr/;"73"+/u/<-2pt>
	\ar@{-}"64"+/dr/;"75"+/u/<-2pt>
	\ar@{-}"65"+/dl/;"75"+/u/<2pt>
	\ar@{-}"65"+/dr/;"76"+/u/<-2pt>
	\end{xy},
	$$
where numbers connected by solid lines are in the same degree, 
the tops of the $\Pi e_i$ are concentrated in degree $0$,
and the degree zero parts are denoted by bold numbers.

Let $w=s_1s_2s_3s_1s_2s_1$.
This is a $c$-sortable element, where $c=c^{(0)}=s_1s_2s_3$, $c^{(1)}=s_1s_2$, and $c^{(3)}=s_{1}$.
Then we have a graded algebra, $\Pi_w=\Pi_we_1\oplus\Pi_we_2\oplus\Pi_we_3$, where
	$$
	\Pi_{w} e_{1}=\begin{xy} (0,0)="O",
	"O"+<0cm,1cm>="11"*{{\bf 1}},
	"11"+<-0.3cm,-0.5cm>="21"*{2},
	"21"+<-0.3cm,-0.5cm>="31"*{3},
	"31"+<-0.3cm,-0.5cm>="41"*{1},
	"21"+/r0.6cm/="22"*{3},
	"31"+/r0.6cm/="32"*{1},
	"32"+/r0.6cm/="33"*{2},
	"41"+/r0.6cm/="42",
	"42"+/r0.6cm/="43",
	"43"+/r0.6cm/="44"*{1},
	\ar@{-}"21"+/dr/;"32"+/u/<-2pt>
	\ar@{-}"22"+/dl/;"32"+/u/<2pt>
	\ar@{-}"22"+/dr/;"33"+/u/<-2pt>
	\ar@{-}"31"+/dl/;"41"+/u/<2pt>
	\ar@{-}"33"+/dr/;"44"+/u/<-2pt>
	\end{xy},
\qquad
	\Pi_{w}e_{2}=\begin{xy} (0,0)="O",
	"O"+<0cm,1cm>="11"*{{\bf 2}},
	"11"+<-0.3cm,-0.5cm>="21"*{3},
	"21"+<-0.3cm,-0.5cm>="31"*{1},
	"31"+<-0.3cm,-0.5cm>="41",
	"41"+<-0.3cm,-0.5cm>="51",
	"21"+/r0.6cm/="22"*{{\bf 1}},
	"31"+/r0.6cm/="32"*{2},
	"32"+/r0.6cm/="33"*{3},
	"41"+/r0.6cm/="42",
	"42"+/r0.6cm/="43"*{1},
	"43"+/r0.6cm/="44"*{2},
	"51"+/r0.6cm/="52",
	"52"+/r0.6cm/="53",
	"53"+/r0.6cm/="54",
	"54"+/r0.6cm/="55"*{1},
	\ar@{-}"11"+/dr/;"22"+/u/<-2pt>
	\ar@{-}"21"+/dl/;"31"+/u/<2pt>
	\ar@{-}"21"+/dr/;"32"+/u/<-2pt>
	\ar@{-}"32"+/dr/;"43"+/u/<-2pt>
	\ar@{-}"33"+/dl/;"43"+/u/<2pt>
	\ar@{-}"33"+/dr/;"44"+/u/<-2pt>
	\ar@{-}"44"+/dr/;"55"+/u/<-2pt>
	\end{xy},
\qquad
	\Pi_{w}e_{3}=\begin{xy} (0,0)="O",
	"O"+<0cm,0.7cm>="11"*{{\bf 3}},
	"11"+<-0.3cm,-0.5cm>="21"*{{\bf 1}},
	"21"+<-0.3cm,-0.5cm>="31",
	"21"+/r0.6cm/="22"*{{\bf 2}},
	"31"+/r0.6cm/="32",
	"32"+/r0.6cm/="33"*{{\bf 1}},
	\ar@{-}"11"+/dl/;"21"+/u/<2pt>
	\ar@{-}"11"+/dr/;"22"+/u/<-2pt>
	\ar@{-}"22"+/dr/;"33"+/u/<-2pt>
	\end{xy}.
	$$
\end{example}
\section{Graded structure of $I_w$ and $\Pi_w$}\label{gradedstrof}
In this section, we prove some basic properties of gradings of $I_w$ and $\Pi_w$.
The main result in this section is Proposition \ref{pii}.
We also recall some results from \cite{AIRT} which will be used later.
Let $c=s_{u_1}s_{u_2}\cdots s_{u_n}$ be a Coxeter element in $W=W_{Q}$ satisfying $e_{u_j}(kQ)e_{u_i}=0$ for $i < j$.
\begin{lemma}\cite[Lemma 2.1]{AIRT}\label{airtlem2.1}
Let $Q^{\prime}$ be a full subquiver of $Q$ and  $w$ an element in $W_{Q^{\prime}}\subset W$.
Then we have $\Pi/I_{w} = \Pi^{\prime}/I_{w}^{\prime}$ as graded algebras,
where $\Pi^{\prime}$ is a preprojective algebra of $Q^{\prime}$ and $I^{\prime}_{w}$ is the ideal of $\Pi^{\prime}$ associated with $w$.
\end{lemma}
We first calculate the ideal $I_{w}$ and the algebra $\Pi_{w}$ when $w$ is a subsequence of a Coxeter element.
\begin{lemma}\label{iw}
Let $w\in W$ be a subsequence of a Coxeter element and $Q^{\prime}$ the full subquiver of $Q$ whose set of vertices is $\supp(w)$.
We denote by $\Pi^{\prime}$ the preprojective algebra of $Q^{\prime}$ and $I^{\prime}_{w}$ the ideal of $\Pi^{\prime}$ associated with $w$.
Then the following holds.
\begin{itemize}
\item[(a)]
We have $\Pi/I_{w} = \Pi^{\prime}/I_{w}^{\prime}= kQ^{\prime}$.
\item[(b)]
$(I_w)_{\geq 1}=\Pi_{\geq 1}$.
\item[(c)]
$(I_{w})_{0}$ is the ideal of $kQ$ generated by idempotents $\{e_{u}\mid u\in Q_0\setminus \supp(w)\}$.
\end{itemize}
\end{lemma}
\begin{proof}
(\rm a)
By Lemma \ref{airtlem2.1}, we have $\Pi/I_{w} = \Pi^{\prime}/I_{w}^{\prime}$.
Since $w$ is a subsequence of a Coxeter element,
$w$ is a Coxeter element of $W_{Q^{\prime}}$.
Then, by \cite[Proposition I\hspace{-.1em}I\hspace{-.1em}I. 3.2]{BIRSc}, 
we have $\Pi^{\prime}/I_{w}^{\prime}= kQ^{\prime}$.

(\rm b)
By (a), we have $(\Pi/I_{w})_{0}=(\Pi^{\prime}/I_{w}^{\prime})_{0}= kQ^{\prime}$.
This means that $(I_w)_{\geq 1}=\Pi_{\geq 1}$.

(\rm c)
Since $kQ^{\prime}=(\Pi/I_{w})_{0}=\Pi_{0}/(I_{w})_{0}=kQ/(I_{w})_{0}$ holds, $(I_w)_{0}$ the ideal generated by the vertices in $Q_0\setminus \supp(w)$.
\end{proof}
Then we describe the grading of $I_{w}$ for a $c$-sortable element $w$.
\begin{lemma}\label{ic0}
Let $w=c^{(0)}c^{(1)}\cdots c^{(m)}\in W$ be a $c$-sortable element.
Then we have 
$(I_{c^{(i)}}I_{c^{(i+1)}})_0=(I_{c^{(i)}})_0$ for all $0 \leq i \leq m-1$.
\end{lemma}
\begin{proof}
Since $\Pi_{w}$ is positively graded, we have $(I_{c^{(i)}}I_{c^{(i+1)}})_0=(I_{c^{(i)}})_0(I_{c^{(i+1)}})_0$.
By Lemma \ref{iw}, $(I_{c^{(i)}})_0$ and $(I_{c^{(i+1)}})_0$ are generated by idempotents $\{ e_{v} \mid v \in Q_0\setminus\supp(c^{(i)}) \}$ and $\{ e_{v} \mid v \in Q_0\setminus\supp(c^{(i+1)}) \}$, respectively.
Since $w$ is a $c$-sortable element, we have $\supp(c^{(i+1)})\subset \supp(c^{(i)})$.
Therefore we have $(I_{c^{(i)}})_0(I_{c^{(i+1)}})_0=(I_{c^{(i)}})_0$.
\end{proof}
\begin{lemma}\label{iwi}
For any $c$-sortable element $w=c^{(0)}c^{(1)}\cdots c^{(m)}\in W$, we have 
\begin{align*}
(I_{w})_i=\begin{cases}
			(I_{c^{(0)}c^{(1)}\cdots c^{(i)}})_i & 0\leq i\leq m.\\
			\Pi_{i} & m+1\leq i.
		\end{cases}
\end{align*}
In particular, we have $(\Pi/I_{w})_{\geq m+1}=0$.
\end{lemma}
\begin{proof}
We first show that $(I_{w})_{\geq m+1}=\Pi_{\geq m+1}$.
Since $\Pi$ is generated by $\Pi_{1}$ as a $\Pi_{0}$-algebra, we have $\Pi_{\geq m+1}=\prod_{j=0}^{m}(\Pi_{\geq 1})$.
By Lemma \ref{iw} (b), the equation $\Pi_{\geq 1}=(I_{c^{(j)}})_{\geq 1}$ holds for any $0\leq j \leq m$.
Thus we have \[(I_{w})_{\geq m+1}\subset\Pi_{\geq m+1}=\prod_{j=0}^{m}\Pi_{\geq 1}=\prod_{j=0}^{m}(I_{c^{(j)}})_{\geq 1}\subset(I_{w})_{\geq m+1}.\]
Therefore we have $(I_{w})_{\geq m+1}=\Pi_{\geq m+1}$.

Assume that $0\leq i \leq m-1$.
We show that $(I_{w})_{i}=(I_{c^{(0)}c^{(1)}\cdots c^{(m-1)}})_{i}$.
Since $I_{w}\subset I_{c^{(0)}c^{(1)}\cdots c^{(m-1)}}$, 
we have $(I_{w})_i\subset (I_{c^{(0)}c^{(1)}\cdots c^{(m-1)}})_i$.
Conversely, we show that 
\[(I_{c^{(0)}c^{(1)}\cdots c^{(m-1)}})_i\subset (I_{w})_i.\]
In general, we have
\begin{align}\label{gradeiw}
(I_{c^{(0)}c^{(1)}\cdots c^{(m-1)}})_i=\sum_{b_0+b_1+\cdots +b_{m-1}=i}(I_{c^{(0)}})_{b_0}(I_{c^{(1)}})_{b_1}\cdots(I_{c^{(m-1)}})_{b_{m-1}}.
\end{align}
Since $(I_w)_i=0$ for any $i<0$ and (\ref{gradeiw}),
it is enough to show that 
\[(I_{c^{(0)}})_{a_0}(I_{c^{(1)}})_{a_1}\cdots(I_{c^{(m-1)}})_{a_{m-1}}\subset (I_w)_i,\]
for any non-negative integers $a_0,a_1,\ldots,a_{m-1}$ satisfying $\sum_{j=0}^{m-1}a_{j}=i$ .
Since $a_{0},\ldots,a_{m-1}$ are non-negative and $i\leq m-1$, at least one of them must be zero.
Let $j$ be the largest integer satisfying $a_{j}=0$.
Then we have
\begin{align*}
&(I_{c^{(0)}})_{a_0}\cdots(I_{c^{(j)}})_{a_{j}}(I_{c^{(j+1)}})_{a_{j+1}}\cdots(I_{c^{(m-1)}})_{a_{m-1}} \\
&=(I_{c^{(0)}})_{a_0}\cdots(I_{c^{(j)}})_{a_{j}}(\Pi_{a_{j+1}})\cdots(\Pi_{a_{m-1}})\\
&=(I_{c^{(0)}})_{a_0}\cdots(I_{c^{(j)}})_{a_{j}}(I_{c^{(j+1)}})_{0}(\Pi_{a_{j+1}})\cdots(\Pi_{a_{m-1}})\\
&=(I_{c^{(0)}})_{a_0}\cdots(I_{c^{(j)}})_{a_{j}}(I_{c^{(j+1)}})_{0}(I_{c^{(j+2)}})_{a_{j+1}}\cdots(I_{c^{(m)}})_{a_{m-1}}\\
&\subset(I_{w})_{i},
\end{align*}
where the first and the third equations come form Lemma \ref{iw} (b), and the second equation comes from Lemma \ref{ic0}.
Therefore we have $(I_{c^{(0)}c^{(1)}\cdots c^{(m-1)}})_i\subset (I_{w})_i$ for $0\leq i \leq m-1$.
By using this equation repeatedly, we have the assertion.
\end{proof}
Now we describe the grading of $\Pi_w$ for a $c$-sortable element $w$.
For an element $w$ in $W$,
let $Q^{(1)}$ be the full subquiver of $Q$ whose set of vertices is $\supp(w)$.
\begin{proposition}\label{pii}
Let $w=c^{(0)}c^{(1)}\cdots c^{(m)}\in W$ be a $c$-sortable element and $i\leq m$.
Then we have $(\Pi_w)_{\leq i}=(\Pi_{c^{(0)}c^{(1)}\cdots c^{(i)}})_{\leq i}=\Pi_{c^{(0)}c^{(1)}\cdots c^{(i)}}$.
In particular, we have $(\Pi_{w})_{0}=\Pi_{c^{(0)}}=kQ^{(1)}$.
\end{proposition}
\begin{proof}
By Lemma \ref{iwi}, we have the following commutative diagram.
	$$
	\xymatrix{0 \ar[r] & (I_w)_{\leq i} \ar[r] \ar@{=}[d] & \Pi_{\leq i} \ar[r] \ar@{=}[d] & (\Pi_w)_{\leq i} \ar[r] \ar^{\simeq}[d] & 0\  \\
	0 \ar[r] & (I_{c^{(0)}c^{(1)}\cdots c^{(i)}})_{\leq i} \ar[r] & \Pi_{\leq i} \ar[r] & (\Pi /I_{c^{(0)}c^{(1)}\cdots c^{(i)}})_{\leq i}\ar[r] & 0.}
	$$
Therefore we have an equality $(\Pi_w)_{\leq i} = (\Pi_{c^{(0)}c^{(1)}\cdots c^{(i)}})_{\leq i}$.
The equality $(\Pi_{c^{(0)}c^{(1)}\cdots c^{(i)}})_{\leq i}=\Pi_{c^{(0)}c^{(1)}\cdots c^{(i)}}$ comes from Lemma \ref{iwi}.
If $i=0$, then we have $(\Pi_{w})_{0}=\Pi_{c^{(0)}}=kQ^{(1)}$, where the second equality comes from Lemma \ref{iw} (a).
\end{proof}
The following proposition is important to show Theorem \ref{thm}.
\begin{proposition}\label{tran}
Let $w=s_{u_1}\cdots s_{u_l}=c^{(0)}c^{(1)}\cdots c^{(m)}$ be a $c$-sortable element.
For any integer $i$ and $X\in \Sub^{\mathbb{Z}}\Pi_w$,  
we have $X_{\geq i}, X_{\leq i}, X_{i} \in \Sub^{\mathbb{Z}}\Pi_w$.
\end{proposition}
\begin{proof}
Since $X_{\geq i}$ is a submodule of $X$, we have $X_{\geq i} \in \Sub^{\mathbb{Z}}\Pi_w$.

By Proposition \ref{birs} (e), we have $\Pi/I_{u_1\cdots u_j}\in \Sub^{\mathbb{Z}}\Pi_w$ for any $1 \leq j \leq l$.
Therefore, by Proposition \ref{pii}, we have $(\Pi_{w})_{\leq i}\in\Sub^{\mathbb{Z}}\Pi_w$ for any integer $i$.
Clearly, the functor $X \mapsto X_{\leq i}$ preserves injective morphisms.
Therefore we have $X_{\leq i} \in \Sub^{\mathbb{Z}}\Pi_w$.
Since $X_{\geq i} \in \Sub^{\mathbb{Z}}\Pi_w$, $X_{i}=(X_{\geq i})_{\leq i}\in \Sub^{\mathbb{Z}}\Pi_w$ holds.
\end{proof}
Next we recall the result of \cite{AIRT}.
For a reduced expression $w=s_{u_1}\cdots s_{u_l}$ and $1 \leq i\leq l$, we define a $\Pi_w$-module $L_w^{i}$ by $L_{w}^{1}:=\Pi/I_{u_{1}}$ and
	\begin{align*}
	L_w^{i}:=\frac{I_{u_1\cdots u_{i-1}}}{I_{u_1\cdots u_i}},
	\end{align*}
for $i\geq 2$.
\begin{proposition}\cite[Proposition 1.3]{AIRT}\label{airtprop}
We have equalities
	\begin{align*}
	L_w^i=L_w^ie_{u_i}=\frac{I_{u_1\cdots u_{j}}}{I_{u_1\cdots u_i}}e_{u_i},
	\end{align*}
where $j$ is the largest integer satisfying $j<i$ and $u_j=u_i$.
If such an integer $j$ does not appear in $u_1,\cdots,u_{i-1}$, then $L_w^i=(\Pi/I_{u_1\cdots u_i})e_{u_i}$.
\end{proposition}
We use the following notation.
Let $w=s_{u_1}\cdots s_{u_l}$ be a reduced expression.
For any $u\in \supp(w)$, let 
\[
p_{u}=\Max\{1 \leq j \leq l \mid u_{j}=u\}.
\]
For $1\leq i \leq l$, let
\[
m_i=\sharp\{ 1\leq j \leq i-1\mid u_j=u_i\}.
\]
Note that, if $w=s_{u_1}\cdots s_{u_l}=c^{(0)}c^{(1)}\cdots c^{(m)}$ is a $c$-sortable element, 
then we have $m_{p_{u}}=\Max\{j \mid u \in \supp(c^{(j)})\}$ for any $u\in \supp(w)$.
Using $L_w^i$, we have the following information on $\Pi_we_u$.
\begin{lemma}\label{c}
Let $w=s_{u_1}\cdots s_{u_l}=c^{(0)}c^{(1)}\cdots c^{(m)}$ be a $c$-sortable element.
Then, for any $u\in \supp(w)$ and any integer $i\geq m_{p_{u}}$, we have 
	\begin{align}
	(\Pi_we_u)_{i} = \begin{cases}
				L_{w}^{p_{u}} & i = m_{p_{u}}, \\
				0 & m_{p_{u}}+1 \leq i.
			\end{cases}\notag
	\end{align}
\end{lemma}
\begin{proof}
Since $I_{w}e_{u}=(I_{c^{(0)}c^{(1)}\cdots c^{(m_{p_{u}})}})e_{u}$, we have $\Pi_we_u=(\Pi_{c^{(0)}c^{(1)}\cdots c^{(m_{p_{u}})}})e_{u}$.
Thus, by Lemma \ref{iwi}, we have $(\Pi_we_u)_i=0$ for $m_{p_{u}}+1 \leq i$.

If $i=m_{p_{u}}$, we have
\begin{align*}
(\Pi_{w}e_{u})_{i} & = \Ker \left( (\Pi_{w}e_{u})_{\leq i} \to (\Pi_{w}e_{u})_{\leq i-1} \right) \\
			& = \Ker \left(  \frac{\Pi}{I_{c^{(0)}c^{(1)}\cdots c^{(i)}}}e_u \to \frac{\Pi}{I_{c^{(0)}c^{(1)}\cdots c^{(i-1)}}}e_u \right) \\
			& = \frac{I_{c^{(0)}c^{(1)}\cdots c^{(i-1)}}}{I_{c^{(0)}c^{(1)}\cdots c^{(i)}}}e_u,
\end{align*}
where the second equality comes from Proposition \ref{pii}.
Since $I_{c^{(0)}c^{(1)}\cdots c^{(i)}}e_{u}=I_{u_1\cdots u_{p_{u}}}e_{u}$, we have the desired equality.
\end{proof}
The next theorem is one of the main results in \cite{AIRT}, and important in this paper.
We use Theorem \ref{airt} to prove Proposition \ref{M-generate}.
For an element $w$ in $W$,
let $Q^{(1)}$ be the full subquiver of $Q$ whose set of vertices is $\supp(w)$.
\begin{theorem}\label{airt}
Let $w=s_{u_1}\cdots s_{u_l}=c^{(0)}c^{(1)}\cdots c^{(m)}$ be a $c$-sortable element.
Then \[T=\bigoplus_{u\in Q_0^{(1)}}L_{w}^{p_{u}}=\bigoplus_{u\in Q_0^{(1)}}(\Pi_we_u(m_{p_{u}}))_{0}\] is a tilting $kQ^{(1)}$-module.
\end{theorem}
\begin{proof}
$T=\bigoplus_{u\in Q_0^{(1)}}L_{w}^{p_{u}}$ is a tilting $kQ^{(1)}$-module by \cite[Theorem 3.11]{AIRT}.
Moreover $T=\bigoplus_{u\in Q_0^{(1)}}(\Pi_we_u(m_{p_{u}}))_{0}$ holds by Lemma \ref{c}.
\end{proof}
We give one example which illustrates the tilting module of Theorem \ref{airt}.
\begin{example}
Let $Q$ be a quiver
	\begin{xy} (0,0)="O",
	"O"+<0cm,0.35cm>="11"*{1},
	"11"+<-0.8cm,-0.7cm>="21"*{2},
	"21"+/r1.6cm/="22"*{3},
	
	\ar"11"+/dl/;"21"+/u/
	\ar"21"+/r/;"22"+/l/
	\ar"11"+/dr/;"22"+/u/
	\end{xy} and $w=s_1s_2s_3s_1s_2s_1$.
This is a $c$-sortable element.
Then we have a graded algebra $\Pi_w=\Pi_we_1\oplus\Pi_we_2\oplus\Pi_we_3$,
	\[
	\begin{xy} (0,0)="O",
	"O"+<0cm,1cm>="11"*{{\bf 1}},
	"11"+<-0.3cm,-0.5cm>="21"*{2},
	"21"+<-0.3cm,-0.5cm>="31"*{3},
	"31"+<-0.3cm,-0.5cm>="41"*{1},
	"21"+/r0.6cm/="22"*{3},
	"31"+/r0.6cm/="32"*{1},
	"32"+/r0.6cm/="33"*{2},
	"41"+/r0.6cm/="42",
	"42"+/r0.6cm/="43",
	"43"+/r0.6cm/="44"*{1},
	
	\ar@{-}"21"+/dr/;"32"+/u/<-2pt>
	\ar@{-}"22"+/dl/;"32"+/u/<2pt>
	\ar@{-}"22"+/dr/;"33"+/u/<-2pt>
	\ar@{-}"31"+/dl/;"41"+/u/<2pt>
	\ar@{-}"33"+/dr/;"44"+/u/<-2pt>
	\end{xy}
\qquad
	\begin{xy} (0,0)="O",
	"O"+<0cm,1cm>="11"*{{\bf 2}},
	"11"+<-0.3cm,-0.5cm>="21"*{3},
	"21"+<-0.3cm,-0.5cm>="31"*{1},
	"31"+<-0.3cm,-0.5cm>="41",
	"41"+<-0.3cm,-0.5cm>="51",
	"21"+/r0.6cm/="22"*{{\bf 1}},
	"31"+/r0.6cm/="32"*{2},
	"32"+/r0.6cm/="33"*{3},
	"41"+/r0.6cm/="42",
	"42"+/r0.6cm/="43"*{1},
	"43"+/r0.6cm/="44"*{2},
	"51"+/r0.6cm/="52",
	"52"+/r0.6cm/="53",
	"53"+/r0.6cm/="54",
	"54"+/r0.6cm/="55"*{1},
	
	\ar@{-}"11"+/dr/;"22"+/u/<-2pt>
	\ar@{-}"21"+/dl/;"31"+/u/<2pt>
	\ar@{-}"21"+/dr/;"32"+/u/<-2pt>
	\ar@{-}"32"+/dr/;"43"+/u/<-2pt>
	\ar@{-}"33"+/dl/;"43"+/u/<2pt>
	\ar@{-}"33"+/dr/;"44"+/u/<-2pt>
	\ar@{-}"44"+/dr/;"55"+/u/<-2pt>
	\end{xy}
\qquad
	\begin{xy} (0,0)="O",
	"O"+<0cm,0.7cm>="11"*{{\bf 3}},
	"11"+<-0.3cm,-0.5cm>="21"*{{\bf 1}},
	"21"+<-0.3cm,-0.5cm>="31",
	"21"+/r0.6cm/="22"*{{\bf 2}},
	"31"+/r0.6cm/="32",
	"32"+/r0.6cm/="33"*{{\bf 1}},
	
	\ar@{-}"11"+/dl/;"21"+/u/<2pt>
	\ar@{-}"11"+/dr/;"22"+/u/<-2pt>
	\ar@{-}"22"+/dr/;"33"+/u/<-2pt>
	\end{xy}.
	\]
We have 
\[L_{w}^{1}=1,
	\qquad 
L_{w}^{2}=
	\begin{xy}(0,0)="00",
	"00"+<-0.3cm,0.25cm>="11"*{2},
	"11"+<-0.3cm,-0.5cm>="21",
	"21"+/r0.6cm/="22"*{1},
	\ar@{-}"11"+/dr/;"22"+/u/<-2pt>
	\end{xy},
	\qquad
L_{w}^{3}=
	\begin{xy} (0,0)="O",
	"O"+<0cm,0.7cm>="11"*{3},
	"11"+<-0.3cm,-0.5cm>="21"*{1},
	"21"+<-0.3cm,-0.5cm>="31",
	"21"+/r0.6cm/="22"*{2},
	"31"+/r0.6cm/="32",
	"32"+/r0.6cm/="33"*{1},
	
	\ar@{-}"11"+/dl/;"21"+/u/<2pt>
	\ar@{-}"11"+/dr/;"22"+/u/<-2pt>
	\ar@{-}"22"+/dr/;"33"+/u/<-2pt>
	\end{xy},\]
\[L_{w}^{4}=
	\begin{xy} (0,0)="O",
	"O"+<0cm,1cm>="11",
	"11"+<-0.3cm,-0.5cm>="21"*{2},
	"21"+<-0.3cm,-0.5cm>="31",
	"31"+<-0.3cm,-0.5cm>="41",
	"21"+/r0.6cm/="22"*{3},
	"31"+/r0.6cm/="32"*{1},
	"32"+/r0.6cm/="33"*{2},
	"41"+/r0.6cm/="42",
	"42"+/r0.6cm/="43",
	"43"+/r0.6cm/="44"*{1},
	
	\ar@{-}"21"+/dr/;"32"+/u/<-2pt>
	\ar@{-}"22"+/dl/;"32"+/u/<2pt>
	\ar@{-}"22"+/dr/;"33"+/u/<-2pt>
	\ar@{-}"33"+/dr/;"44"+/u/<-2pt>
	\end{xy},
	\qquad
L_{w}^{5}=
	\begin{xy} (0,0)="O",
	"O"+<0cm,1cm>="11",
	"11"+<-0.3cm,-0.5cm>="21"*{3},
	"21"+<-0.3cm,-0.5cm>="31"*{1},
	"31"+<-0.3cm,-0.5cm>="41",
	"41"+<-0.3cm,-0.5cm>="51",
	"21"+/r0.6cm/="22",
	"31"+/r0.6cm/="32"*{2},
	"32"+/r0.6cm/="33"*{3},
	"41"+/r0.6cm/="42",
	"42"+/r0.6cm/="43"*{1},
	"43"+/r0.6cm/="44"*{2},
	"51"+/r0.6cm/="52",
	"52"+/r0.6cm/="53",
	"53"+/r0.6cm/="54",
	"54"+/r0.6cm/="55"*{1},
	
	\ar@{-}"21"+/dl/;"31"+/u/<2pt>
	\ar@{-}"21"+/dr/;"32"+/u/<-2pt>
	\ar@{-}"32"+/dr/;"43"+/u/<-2pt>
	\ar@{-}"33"+/dl/;"43"+/u/<2pt>
	\ar@{-}"33"+/dr/;"44"+/u/<-2pt>
	\ar@{-}"44"+/dr/;"55"+/u/<-2pt>
	\end{xy},
	\quad
L_{w}^{6}=
	\begin{xy} (0,0)="00",
	"00"+<0.3cm,0.25cm>="11"*{3},
	"11"+<-0.3cm,-0.5cm>="21"*{1},
	\ar@{-}"11"+/dl/;"21"+/u/
	\end{xy}.
\]
By Theorem \ref{airt}, $L_{w}^{3}\oplus L_{w}^{5}\oplus L_{w}^{6}$ is a tilting $kQ$-module.
\if()
Then we have a graded algebra $\Pi=\Pi e_1\oplus\Pi e_2\oplus\Pi e_3$, and these are represented by their radical filtrations
	$$
	\begin{xy} (0,0)="11"*{1},
	"11"+<-0.3cm,-0.5cm>="21"*{2},
	"21"+<-0.3cm,-0.5cm>="31"*{3},
	"31"+<-0.3cm,-0.5cm>="41"*{1},
	"41"+<-0.3cm,-0.5cm>="51"*{2},
	"51"+<-0.3cm,-0.5cm>="61"*{3},
	"61"+<-0.3cm,-0.5cm>="71",
	"21"+/r0.6cm/="22"*{3},
	"31"+/r0.6cm/="32"*{1},
	"32"+/r0.6cm/="33"*{2},
	"41"+/r0.6cm/="42"*{2},
	"42"+/r0.6cm/="43"*{3},
	"43"+/r0.6cm/="44"*{1},
	"51"+/r0.6cm/="52"*{3},
	"52"+/r0.6cm/="53"*{1},
	"53"+/r0.6cm/="54"*{2},
	"54"+/r0.6cm/="55"*{3},
	"61"+/r0.6cm/="62"*{1},
	"62"+/r0.6cm/="63"*{2},
	"63"+/r0.6cm/="64"*{3},
	"64"+/r0.6cm/="65"*{1},
	"65"+/r0.6cm/="66"*{2},
	"71"+/r0.6cm/="72",
	"72"+/r0.6cm/="73",
	"73"+/r0.6cm/="74",
	"74"+/r0.6cm/="75",
	"75"+/r0.6cm/="76",
	"76"+/r0.6cm/="77",
	
	\ar@{-}"11"+/dl/;"21"+/u/<2pt>
	\ar@{-}"11"+/dr/;"22"+/u/<-2pt>
	\ar@{-}"21"+/dl/;"31"+/u/<2pt>
	\ar@{-}"32"+/dl/;"42"+/u/<2pt>
	\ar@{-}"32"+/dr/;"43"+/u/<-2pt>
	\ar@{-}"33"+/dl/;"43"+/u/<2pt>
	\ar@{-}"41"+/dl/;"51"+/u/<2pt>
	\ar@{-}"41"+/dr/;"52"+/u/<-2pt>
	\ar@{-}"42"+/dl/;"52"+/u/<2pt>
	\ar@{-}"44"+/dl/;"54"+/u/<2pt>
	\ar@{-}"44"+/dr/;"55"+/u/<-2pt>
	\ar@{-}"51"+/dl/;"61"+/u/<2pt>
	\ar@{-}"53"+/dl/;"63"+/u/<2pt>
	\ar@{-}"53"+/dr/;"64"+/u/<-2pt>
	\ar@{-}"54"+/dl/;"64"+/u/<2pt>
	\ar@{-}"62"+/dl/;"72"+/u/<2pt>
	\ar@{-}"62"+/dr/;"73"+/u/<-2pt>
	\ar@{-}"63"+/dl/;"73"+/u/<2pt>
	\ar@{-}"65"+/dl/;"75"+/u/<2pt>
	\ar@{-}"65"+/dr/;"76"+/u/<-2pt>
	\ar@{-}"66"+/dl/;"76"+/u/<2pt>
	\end{xy}
\qquad
	\begin{xy} (0,0)="11"*{2},
	"11"+<-0.3cm,-0.5cm>="21"*{3},
	"21"+<-0.3cm,-0.5cm>="31"*{1},
	"31"+<-0.3cm,-0.5cm>="41"*{2},
	"41"+<-0.3cm,-0.5cm>="51"*{3},
	"51"+<-0.3cm,-0.5cm>="61"*{1},
	"61"+<-0.3cm,-0.5cm>="71",
	"21"+/r0.6cm/="22"*{1},
	"31"+/r0.6cm/="32"*{2},
	"32"+/r0.6cm/="33"*{3},
	"41"+/r0.6cm/="42"*{3},
	"42"+/r0.6cm/="43"*{1},
	"43"+/r0.6cm/="44"*{2},
	"51"+/r0.6cm/="52"*{1},
	"52"+/r0.6cm/="53"*{2},
	"53"+/r0.6cm/="54"*{3},
	"54"+/r0.6cm/="55"*{1},
	"61"+/r0.6cm/="62"*{2},
	"62"+/r0.6cm/="63"*{3},
	"63"+/r0.6cm/="64"*{1},
	"64"+/r0.6cm/="65"*{2},
	"65"+/r0.6cm/="66"*{3},
	"71"+/r0.6cm/="72",
	"72"+/r0.6cm/="73",
	"73"+/r0.6cm/="74",
	"74"+/r0.6cm/="75",
	"75"+/r0.6cm/="76",
	"76"+/r0.6cm/="77",
	
	\ar@{-}"11"+/dl/;"21"+/u/<2pt>
	\ar@{-}"22"+/dl/;"32"+/u/<2pt>
	\ar@{-}"22"+/dr/;"33"+/u/<-2pt>
	\ar@{-}"31"+/dl/;"41"+/u/<2pt>
	\ar@{-}"31"+/dr/;"42"+/u/<-2pt>
	\ar@{-}"32"+/dl/;"42"+/u/<2pt>
	\ar@{-}"41"+/dl/;"51"+/u/<2pt>
	\ar@{-}"43"+/dl/;"53"+/u/<2pt>
	\ar@{-}"43"+/dr/;"54"+/u/<-2pt>
	\ar@{-}"44"+/dl/;"54"+/u/<2pt>
	\ar@{-}"52"+/dl/;"62"+/u/<2pt>
	\ar@{-}"52"+/dr/;"63"+/u/<-2pt>
	\ar@{-}"53"+/dl/;"63"+/u/<2pt>
	\ar@{-}"55"+/dl/;"65"+/u/<2pt>
	\ar@{-}"55"+/dr/;"66"+/u/<-2pt>
	\ar@{-}"61"+/dl/;"71"+/u/<2pt>
	\ar@{-}"61"+/dr/;"72"+/u/<-2pt>
	\ar@{-}"62"+/dl/;"72"+/u/<2pt>
	\ar@{-}"64"+/dl/;"74"+/u/<2pt>
	\ar@{-}"64"+/dr/;"75"+/u/<-2pt>
	\ar@{-}"65"+/dl/;"75"+/u/<2pt>
	\end{xy}
\qquad
	\begin{xy} (0,0)="11"*{3},
	"11"+<-0.3cm,-0.4cm>="21"*{1},
	"21"+<-0.3cm,-0.4cm>="31"*{2},
	"31"+<-0.3cm,-0.4cm>="41"*{3},
	"41"+<-0.3cm,-0.4cm>="51"*{1},
	"51"+<-0.3cm,-0.4cm>="61"*{2},
	"61"+<-0.3cm,-0.4cm>="71",
	"21"+/r0.6cm/="22"*{2},
	"31"+/r0.6cm/="32"*{3},
	"32"+/r0.6cm/="33"*{1},
	"41"+/r0.6cm/="42"*{1},
	"42"+/r0.6cm/="43"*{2},
	"43"+/r0.6cm/="44"*{3},
	"51"+/r0.6cm/="52"*{2},
	"52"+/r0.6cm/="53"*{3},
	"53"+/r0.6cm/="54"*{1},
	"54"+/r0.6cm/="55"*{2},
	"61"+/r0.6cm/="62"*{3},
	"62"+/r0.6cm/="63"*{1},
	"63"+/r0.6cm/="64"*{2},
	"64"+/r0.6cm/="65"*{3},
	"65"+/r0.6cm/="66"*{1},
	"71"+/r0.6cm/="72",
	"72"+/r0.6cm/="73",
	"73"+/r0.6cm/="74",
	"74"+/r0.6cm/="75",
	"75"+/r0.6cm/="76",
	"76"+/r0.6cm/="77",
	
	\ar@{-}"21"+/dl/;"31"+<0cm,0.15cm>
	\ar@{-}"21"+/dr/;"32"+<0cm,0.15cm>
	\ar@{-}"22"+/dl/;"32"+<0cm,0.15cm>
	\ar@{-}"31"+/dl/;"41"+<0cm,0.15cm>
	\ar@{-}"33"+/dl/;"43"+<0cm,0.15cm>
	\ar@{-}"33"+/dr/;"44"+<0cm,0.15cm>
	\ar@{-}"42"+/dl/;"52"+<0cm,0.15cm>
	\ar@{-}"42"+/dr/;"53"+<0cm,0.15cm>
	\ar@{-}"43"+/dl/;"53"+<0cm,0.15cm>
	\ar@{-}"51"+/dl/;"61"+<0cm,0.15cm>
	\ar@{-}"51"+/dr/;"62"+<0cm,0.15cm>
	\ar@{-}"52"+/dl/;"62"+<0cm,0.15cm>
	\ar@{-}"54"+/dl/;"64"+<0cm,0.15cm>
	\ar@{-}"54"+/dr/;"65"+<0cm,0.15cm>
	\ar@{-}"55"+/dl/;"65"+<0cm,0.15cm>
	\ar@{-}"61"+/dl/;"71"+<0cm,0.15cm>
	\ar@{-}"63"+/dl/;"73"+<0cm,0.15cm>
	\ar@{-}"63"+/dr/;"74"+<0cm,0.15cm>
	\ar@{-}"64"+/dl/;"74"+<0cm,0.15cm>
	\ar@{-}"66"+/dl/;"76"+<0cm,0.15cm>
	\ar@{-}"66"+/dr/;"77"+<0cm,0.15cm>
	\end{xy}
	$$
where numbers connected by solid lines are in the same degree, 
and the tops of the $\Pi e_i$ are concentrated in degree $0$.

This is a $c$-sortable element.
Then we have a graded algebra $\Pi_w=\Pi_we_1\oplus\Pi_we_2\oplus\Pi_we_3$,
	\[
	\begin{xy} (0,0)="11"*{1},
	"11"+<-0.3cm,-0.5cm>="21"*{2},
	"21"+<-0.3cm,-0.5cm>="31"*{3},
	"21"+/r0.6cm/="22"*{3},
	
	\ar@{-}"11"+/dl/;"21"+/u/<2pt>
	\ar@{-}"11"+/dr/;"22"+/u/<-2pt>
	\ar@{-}"21"+/dl/;"31"+/u/<2pt>
	\end{xy}
\qquad
	\begin{xy} (0,0)="11"*{2},
	"11"+<-0.3cm,-0.5cm>="21"*{3},
	"21"+<-0.3cm,-0.5cm>="31"*{1},
	"31"+<-0.3cm,-0.5cm>="41"*{2},
	"41"+<-0.3cm,-0.5cm>="51"*{3},
	"21"+/r0.6cm/="22"*{1},
	"31"+/r0.6cm/="32"*{2},
	"32"+/r0.6cm/="33"*{3},
	"41"+/r0.6cm/="42"*{3},
	
	\ar@{-}"11"+/dl/;"21"+/u/<2pt>
	\ar@{-}"22"+/dl/;"32"+/u/<2pt>
	\ar@{-}"22"+/dr/;"33"+/u/<-2pt>
	\ar@{-}"31"+/dl/;"41"+/u/<2pt>
	\ar@{-}"31"+/dr/;"42"+/u/<-2pt>
	\ar@{-}"32"+/dl/;"42"+/u/<2pt>
	\ar@{-}"41"+/dl/;"51"+/u/<2pt>
	\end{xy}
\qquad
	\begin{xy} (0,0)="11"*{3},
	"11"+<-0.3cm,-0.5cm>="21"*{1},
	"21"+<-0.3cm,-0.5cm>="31"*{2},
	"31"+<-0.3cm,-0.5cm>="41"*{3},
	"21"+/r0.6cm/="22"*{2},
	"31"+/r0.6cm/="32"*{3},
	"32"+/r0.6cm/="33"*{1},
	"41"+/r0.6cm/="42",
	"42"+/r0.6cm/="43",
	"43"+/r0.6cm/="44"*{3},
	
	\ar@{-}"21"+/dl/;"31"+/u/<2pt>
	\ar@{-}"21"+/dr/;"32"+/u/<-2pt>
	\ar@{-}"22"+/dl/;"32"+/u/<2pt>
	\ar@{-}"31"+/dl/;"41"+/u/<2pt>
	\ar@{-}"33"+/dr/;"44"+/u/<-2pt>
	\end{xy}.
	\]\fi
\end{example}
\section{A tilting object in $\underline{\Sub}^{\mathbb{Z}}\Pi_w$}\label{tilting}
In this section, we construct a tilting object in $\underline{\Sub}^{\mathbb{Z}}\Pi_w$ when $w$ is a $c$-sortable element.
A triangle equivalence induced from tilting objects is given in Section \ref{theendomorphism}.
We first recall the definition of tilting objects in triangulated categories (e.g. \cite{IT, Y}).
\begin{definition}\label{deftilting}
Let $\mathcal{T}$ be a triangulated category.
An object $U$ in $\mathcal{T}$ is called a {\it tilting object} if the following holds.
\begin{itemize}
\item $\Hom_{\mathcal{T}}(U,U[j])=0$ for any $j\neq 0$.
\item $\thick U=\mathcal{T}$, where $\thick U$ is the smallest triangulated full subcategory of $\mathcal{T}$ containing $U$ and closed under direct summands.
\end{itemize}
\end{definition}
\begin{definition}\label{def-tilting-c-sort}
For a $c$-sortable elemnt $w=c^{(0)}c^{(1)}\cdots c^{(m)}$, put \[M:=\bigoplus _{i=0}^m(\Pi_{c^{(0)}\cdots c^{(i)}})(i). \]
\end{definition}
Throughout this section, let $w=c^{(0)}c^{(1)}\cdots c^{(m)}$ be a $c$-sortable element and let $M$ as in Definition \ref{def-tilting-c-sort}.
This $M$ belongs to $\Sub^{\mathbb{Z}}\Pi_w$ by Proposition \ref{birs} (e) and (\ref{subz=sub}).
\begin{example}\label{exampleA2tilting}
Let $Q$ be a quiver
	\begin{xy} (0,0)="O",
	"O"+<0cm,0.35cm>="11"*{1},
	"11"+<-0.8cm,-0.7cm>="21"*{2},
	"21"+/r1.6cm/="22"*{3},
	
	\ar"11"+/dl/;"21"+/u/
	\ar"21"+/r/;"22"+/l/
	\ar"11"+/dr/;"22"+/u/
	\end{xy}.
Let $w=s_1s_2s_3s_1s_2s_1$.
This is a $c$-sortable element.
Then we have a graded algebra $\Pi_w=\Pi_we_1\oplus\Pi_we_2\oplus\Pi_we_3$,
	\[
	\begin{xy} (0,0)="O",
	"O"+<0cm,1cm>="11"*{{\bf 1}},
	"11"+<-0.3cm,-0.5cm>="21"*{2},
	"21"+<-0.3cm,-0.5cm>="31"*{3},
	"31"+<-0.3cm,-0.5cm>="41"*{1},
	"21"+/r0.6cm/="22"*{3},
	"31"+/r0.6cm/="32"*{1},
	"32"+/r0.6cm/="33"*{2},
	"41"+/r0.6cm/="42",
	"42"+/r0.6cm/="43",
	"43"+/r0.6cm/="44"*{1},
	
	\ar@{-}"21"+/dr/;"32"+/u/<-2pt>
	\ar@{-}"22"+/dl/;"32"+/u/<2pt>
	\ar@{-}"22"+/dr/;"33"+/u/<-2pt>
	\ar@{-}"31"+/dl/;"41"+/u/<2pt>
	\ar@{-}"33"+/dr/;"44"+/u/<-2pt>
	\end{xy}
\qquad
	\begin{xy} (0,0)="O",
	"O"+<0cm,1cm>="11"*{{\bf 2}},
	"11"+<-0.3cm,-0.5cm>="21"*{3},
	"21"+<-0.3cm,-0.5cm>="31"*{1},
	"31"+<-0.3cm,-0.5cm>="41",
	"41"+<-0.3cm,-0.5cm>="51",
	"21"+/r0.6cm/="22"*{{\bf 1}},
	"31"+/r0.6cm/="32"*{2},
	"32"+/r0.6cm/="33"*{3},
	"41"+/r0.6cm/="42",
	"42"+/r0.6cm/="43"*{1},
	"43"+/r0.6cm/="44"*{2},
	"51"+/r0.6cm/="52",
	"52"+/r0.6cm/="53",
	"53"+/r0.6cm/="54",
	"54"+/r0.6cm/="55"*{1},
	
	\ar@{-}"11"+/dr/;"22"+/u/<-2pt>
	\ar@{-}"21"+/dl/;"31"+/u/<2pt>
	\ar@{-}"21"+/dr/;"32"+/u/<-2pt>
	\ar@{-}"32"+/dr/;"43"+/u/<-2pt>
	\ar@{-}"33"+/dl/;"43"+/u/<2pt>
	\ar@{-}"33"+/dr/;"44"+/u/<-2pt>
	\ar@{-}"44"+/dr/;"55"+/u/<-2pt>
	\end{xy}
\qquad
	\begin{xy} (0,0)="O",
	"O"+<0cm,0.7cm>="11"*{{\bf 3}},
	"11"+<-0.3cm,-0.5cm>="21"*{{\bf 1}},
	"21"+<-0.3cm,-0.5cm>="31",
	"21"+/r0.6cm/="22"*{{\bf 2}},
	"31"+/r0.6cm/="32",
	"32"+/r0.6cm/="33"*{{\bf 1}},
	
	\ar@{-}"11"+/dl/;"21"+/u/<2pt>
	\ar@{-}"11"+/dr/;"22"+/u/<-2pt>
	\ar@{-}"22"+/dr/;"33"+/u/<-2pt>
	\end{xy}.
	\]
and
\[
M={\bf 1} 
\oplus 
	\begin{xy} (0,0)="O",
	"O"+<0.3cm,0.25cm>="11"*{{\bf 2}},
	"11"+<-0.3cm,-0.5cm>="21",
	"21"+/r0.6cm/="22"*{{\bf 1}},
	
	\ar@{-}"11"+/dr/;"22"+/u/<-2pt>
	\end{xy}
\oplus \left(
	\begin{xy} (0,0)="O",
	"O"+<0cm,0.7cm>="11"*{1},
	"11"+<-0.3cm,-0.5cm>="21"*{\bf 2},
	"21"+<-0.3cm,-0.5cm>="31",
	"31"+<-0.3cm,-0.5cm>="41",
	"21"+/r0.6cm/="22"*{\bf 3},
	"31"+/r0.6cm/="32"*{\bf 1},
	"32"+/r0.6cm/="33"*{\bf 2},
	"41"+/r0.6cm/="42",
	"42"+/r0.6cm/="43",
	"43"+/r0.6cm/="44"*{\bf 1},
	
	\ar@{-}"21"+/dr/;"32"+/u/<-2pt>
	\ar@{-}"22"+/dl/;"32"+/u/<2pt>
	\ar@{-}"22"+/dr/;"33"+/u/<-2pt>
	\ar@{-}"33"+/dr/;"44"+/u/<-2pt>
	\end{xy}
	\right)
\]
in $\underline{\Sub}^{\mathbb{Z}}\Pi_w$,
where the graded projective $\Pi_w$-modules are removed, 
and the degree zero parts are denoted by bold numbers.
\end{example}
The following proposition follows from Proposition \ref{pii}.
\begin{proposition}\label{Mleq}
$M=M_{\leq 0}$.
\end{proposition}
\begin{proof}
We have $M=\bigoplus _{i=0}^m(\Pi_{c^{(0)}\cdots c^{(i)}})_{\leq i}(i)=\bigoplus _{i=0}^m(\Pi_{c^{(0)}\cdots c^{(i)}})(i)_{\leq 0}=M_{\leq 0}$.
\end{proof}
By the following two propositions, we show that this $M$ satisfies the axioms of tilting objects.
Note that, by Lemma \ref{iwi}, $(\Pi_{w})_{\leq i}=\Pi_{w}$ holds for $i\geq m$, and therefore, we have
\[M=\bigoplus _{i=0}^m(\Pi_{c^{(0)}\cdots c^{(i)}})_{\leq i}(i)=\bigoplus _{i\geq 0}(\Pi_{w})_{\leq i}(i)=\bigoplus _{i\geq 0}(\Pi_{w}(i))_{\leq 0}\]
in $\underline{\Sub}^{\mathbb{Z}}\Pi_w$ by Proposition \ref{pii}.
\begin{proposition}\label{M-orthogonal}
We have $\underline{\Hom}_{\Pi_{w}}^{\mathbb{Z}}(M,M[j])=0$ for any $j\neq 0$.
\end{proposition}
\begin{proof}
For any $0\leq i$, we have a short exact sequence,
	\begin{center}
	$0\to (\Pi_w)(i)_{\geq 1}\to (\Pi_w)(i)\to (\Pi_w)(i)_{\leq 0}\to 0$.
	\end{center}
Since $((\Pi_w)(i)_{\geq 1})_{\leq 0}=0$,
we have \[(\Omega M)_{\leq 0}=\bigoplus_{i\geq 0}\bigl(\Omega\left( \Pi_{w}(i)_{\leq 0} \right)\bigr)_{\leq 0}=\bigoplus_{i\geq 0}\bigl(\Pi_{w}(i)_{\geq 1}\bigr)_{\leq 0}=0.\]
Since $\Pi_w$ is positively graded,
we have $\bigl(\Omega^j (M)\bigr)_{\leq 0}=0$ for $j\geq 1$.
Therefore
	\begin{center}
	$\Hom^{\mathbb{Z}}_{\Pi_w}(M,\Omega^j (M))=0$ and $\Hom^{\mathbb{Z}}_{\Pi_w}(\Omega^j (M),M)=0$
	\end{center}
hold for any $j\geq 1$ by Proposition \ref{Mleq}.
The first equality implies $\underline{\Hom}_{\Pi_{w}}^{\mathbb{Z}}(M,M[-j])=0$ for $j\geq 1$,
and the second equality implies $\underline{\Hom}_{\Pi_{w}}^{\mathbb{Z}}(M,M[j])=0$ for $j\geq 1$.
\end{proof}
Next we prove that $M$ satisfies the second axiom of tilting objects.
Since $(\Pi_w)_0=kQ^{(1)}$ by Proposition \ref{pii}, we regard a  $kQ^{(1)}$-module $X$ as a graded $\Pi_w$-module concentrated in degree $0$.
For an integer $i$, let $\mod^{\leq i}\Pi_{w}$ be the full subcategory of $\mod^{\mathbb{Z}}\Pi_{w}$ of modules $X$ satisfying $X=X_{\leq i}$.
\begin{proposition}\label{M-generate}
We have $\underline{\Sub}^{\mathbb{Z}}\Pi_w=\thick M$.
\end{proposition}
\begin{proof}
Let$X\in\Sub^{\mathbb{Z}}\Pi_w$.
We show that $X\in\thick M$.
By Proposition \ref{tran}, we have $X_{i}\in\Sub^{\mathbb{Z}}\Pi_w$ for any $i\in\mathbb{Z}$.
Since $X$ has a finite filtration $\{ X_{\geq j} \mid j\in\mathbb{Z} \}$, 
it is enough to show that $X_{i}\in\thick M$ for any $i\in\mathbb{Z}$.
Since each $X_{i}$ is a $kQ^{(1)}$-module and the global dimension of $kQ^{(1)}$ is at most one,
it is enough to show that $kQ^{(1)}(i)\in\thick M$ for any $i\in\mathbb{Z}$.

Firstly, we show $kQ^{(1)}(i)\in \thick M$ for any $i\geq 0$ by induction on $i$.
Since $M$ has a direct summand $(\Pi_w)_0=kQ^{(1)}$, we have $kQ^{(1)}\in\thick M$.
Assume $kQ^{(1)}(j)\in \thick M$ for $0\leq j\leq i-1$.
Consider a short exact sequence
	\begin{align}\label{prf1}
	0\to (\Pi_w)_{[1,i]}(i)\to (\Pi_w)_{\leq i}(i)\to (\Pi_w)_0(i)\to 0.
	\end{align}
By taking a finite filtration of $(\Pi_w)_{[1,i]}(i)$ and the inductive hypothesis, 
we conclude that $(\Pi_w)_{[1,i]}(i)\in\thick M$.
Since $(\Pi_w)_{\leq i}(i)$ is a direct summand of M or a graded projective $\Pi_{w}$-module,
we have $kQ^{(1)}(i)=(\Pi_w)_0(i)\in\thick M$ by (\ref{prf1}).
Consequently, we have that $X\in\thick M$ for any $X\in\mod^{\leq 0}\Pi_{w}\cap\Sub^{\mathbb{Z}}\Pi_w$.

Secondly, we show that $kQ^{(1)}(-i)\in \thick M$ for any $i \geq 0$ by induction on $i$.
Assume $kQ^{(1)}(-j)\in \thick M$ for $0 \leq j \leq i-1$.
Thus we have $X\in\thick M$ for any $X\in\mod^{\leq i-1}\Pi_{w}\cap\Sub^{\mathbb{Z}}\Pi_w$. 
By Theorem \ref{airt}, $T=\bigoplus_{u\in Q_{0}^{(1)}}(\Pi_{w}e_{u}(m_{p_{u}}))_{0}$ is a tilting $kQ^{(1)}$-module.
There exists a short exact sequence
	\begin{align*}
	0\to kQ^{(1)}\to T_0\to T_1\to 0,
	\end{align*}
where $T_0,T_1\in \add T$.
Therefore it is enough to show that $T(-i)\in \thick M$.
For each $u\in Q_{0}^{(1)}$, take a short exact sequence 
	\begin{align}\label{prf2}
	0\to Te_{u}(-i)\to \Pi_{w}e_{u}(m_{p_{u}})(-i) \to \Pi_{w}e_{u}(m_{p_{u}})_{\leq -1}(-i)\to 0.
	\end{align}
The second term is a graded projective $\Pi_w$-module.
The third term belongs to $\thick M$ since $\Pi_{w}e_{u}(m_{p_{u}})_{\leq -1}(-i)$ is in $\mod^{\leq i-1}\Pi_{w}$.
Consequently, we have $T(-i)\in \thick M$ by (\ref{prf2}).
\end{proof}
Then we have the main theorem of this section.
\begin{theorem}\label{thm}
For a $c$-sortable element $w=c^{(0)}c^{(1)}\cdots c^{(m)}$,
let $M=\bigoplus _{i=0}^m(\Pi_{c^{(0)}\cdots c^{(i)}})(i)$.
Then $M$ is a tilting object in $\underline{\Sub}^{\mathbb{Z}}\Pi_w$.
\end{theorem}
\begin{proof}
By Propositions \ref{M-orthogonal}, and \ref{M-generate}, $M$ is a tilting object in $\underline{\Sub}^{\mathbb{Z}}\Pi_w$.
\end{proof}
\begin{remark}
It was shown by Yamaura \cite{Y} that, for a finite dimensional self-injective positively graded algebra $A$, the stable category $\underline{\mod}^{\mathbb{Z}}A$ has a tilting object $\bigoplus_{i\geq 0}(A(i))_{\leq 0}$ if $A_0$ has finite global dimension.
Our tilting object $M$ in $\underline{\Sub}^{\mathbb{Z}}\Pi_w$ is an analog of this since $M=\bigoplus _{i\geq 0}(\Pi_{w})_{\leq i}(i)=\bigoplus _{i\geq 0}(\Pi_{w}(i))_{\leq 0}$ holds.
\end{remark}
\section{The endomorphism algebra of the tilting object}\label{theendomorphism}
In this section, we calculate the endomorphism algebra of the tilting object which was constructed in Definition \ref{def-tilting-c-sort}.
The aim of this section is to prove Theorems \ref{usubz=dend} and \ref{endalg}.
Throughout this section, let $Q$ be a finite acyclic quiver.
\subsection{A morphism from $\End_{\Pi_w}^{\mathbb{Z}}(M)$ to $\End_{kQ^{(1)}}(M_0)$}\label{welldef}
Firstly, we give another description of the tilting object which was constructed in Definition \ref{def-tilting-c-sort}.
Throughout this section, we use the following notation.

\begin{definition}\label{def-p-m-M}
Let $w=s_{u_1}s_{u_2}\cdots s_{u_l}$ be a reduced expression of an element $w$ in the Coxeter group of $Q$.
We use the same notation as after Proposition \ref{airtprop}, that is,
\begin{center}
$p_{u}=\Max\{1 \leq j \leq l \mid u_{j}=u\},$ \hspace{0.4cm} for $u \in \supp(w)$,\\
$m_{i} = \sharp\{ 1\leq j \leq i-1\mid u_j=u_i\},$ \hspace{0.4cm} for $1\leq i \leq l$.
\end{center}
Moreover, for $1 \leq i \leq l$, put
\begin{align*}
M^{i}=(\Pi/I_{u_1\cdots u_i})e_{u_i}(m_i), \hspace{0.5cm} M=\bigoplus_{i=1}^{l}M^{i},\\
P=\bigoplus_{u \in \supp(w)}M^{p_{u}}, \hspace{0.5cm} \hspace{0.5cm} T=P_{0}.
\end{align*}
\end{definition}
Note that $P\in \proj^{\mathbb{Z}}\Pi_{w}$ holds since $\Pi_{w}=\bigoplus_{u \in \supp(w)}M^{p_{u}}(-m_{p_{u}})$.
If $w=s_{u_1}s_{u_2}\cdots s_{u_l}=c^{(0)}c^{(1)}\cdots c^{(m)}$ is a $c$-sortable element, then we have an isomorphism
\begin{align}\label{M-in-usubz}
\bigoplus_{i=1}^{l}M^{i}\simeq \bigoplus _{i=0}^m(\Pi_{c^{(0)}\cdots c^{(i)}})(i)
\end{align}
in $\underline{\Sub}^{\mathbb{Z}}\Pi_{w}$.
In fact, for any $1\leq i\leq l$, $M^{i}=(\Pi/I_{c^{(0)}\cdots c^{(m_{i})}})e_{u_{i}}(m_{i})$ holds by Proposition \ref{airtprop}, and for any $0\leq j\leq m$, if $u\in Q_{0}\setminus\supp(c^{(j)})$, then $(\Pi/I_{c^{(0)}\cdots c^{(j)}})e_{u}=\Pi_{w}e_{u}$ holds, which is projective.
Therefore we have an isomorphism (\ref{M-in-usubz}).
As we have shown in Theorem \ref{thm},  $M=\bigoplus_{i=1}^{l}M^{i}$ is a tilting object in $\underline{\Sub}^{\mathbb{Z}}\Pi_w$.

Before starting the calculating of the endomorphism algebra $\underline{\End}_{\Pi_{w}}^{\mathbb{Z}}(M)$, we state a triangle equivalence induced from a tilting object.
Let $\mathcal{T}$ be the stable category of a Frobenius category, 
and assume that $\mathcal{T}$ is Krull-Schmidt.
If there exists a tilting object $U$ in $\mathcal{T}$, then it follows from \cite[(4.3)]{K} that there exists a triangle equivalence
\begin{align}\label{eq-tilting}
\mathcal{T} \simeq {\mathsf K}^{{\rm b}}(\proj \End_{\mathcal{T}}(U)),
\end{align}
where ${\mathsf K}^{{\rm b}}(\proj \End _{\mathcal{T}}(U))$ is the homotopy category of bounded complexes of projective $\End _{\mathcal{T}}(U)$-modules.
In this subsection, we show that the global dimension of $\underline{\End}_{\Pi_{w}}^{\mathbb{Z}}(M)$ is finite and we have the following theorem.
\begin{theorem}\label{usubz=dend}
Let $w=s_{u_1}s_{u_2}\cdots s_{u_l}=c^{(0)}c^{(1)}\cdots c^{(m)}$ be a $c$-sortable element and $M=\bigoplus_{i=1}^{l}M^{i}$ be a tilting object in $\underline{\Sub}^{\mathbb{Z}}\Pi_{w}$.
Then the global dimension of $\underline{\End}_{\Pi_{w}}^{\mathbb{Z}}(M)$ is finite and we have a triangle equivalence \[\underline{\Sub}^{\mathbb{Z}}\Pi_{w}\simeq{\mathsf D}^{{\rm b}}(\underline{\End}_{\Pi_{w}}^{\mathbb{Z}}(M)).\]
\end{theorem}
\begin{proof}
By Proposition \ref{prop-gldim-fin}, the global dimension of $\underline{\End}_{\Pi_{w}}^{\mathbb{Z}}(M)$ is finite.
By Theorem \ref{thm} and a triangle equivalence (\ref{eq-tilting}), we have the assertion.
\end{proof}
We state another theorem of this section.
Looking at the degree zero part of graded modules, we have the following functor
\[\mathbb{F}:=(-)_{0}: \mod^{\mathbb{Z}}\Pi \to \mod kQ.\]
The functor $\mathbb{F}$ induces the following morphism of algebras
	\[
	F:=\mathbb{F}_{M,M} : \End_{\Pi_w}^{\mathbb{Z}}(M) \to \End_{kQ}(M_0)
	\]
given by $F(f)=f|_{M_{0}}$.
Then we claim the following.
\begin{theorem}\label{endalg}
Let $w$ be a $c$-sortable element. The morphism $F$ induces an isomorphism of algebras $\underline{F}:\underline{\End}_{\Pi_w}^{\mathbb{Z}}(M) \xto{\sim} \End_{kQ}(M_0)/[T]$, which makes the following diagram commutative
\[
\xymatrix{
	\End_{\Pi_w}^{\mathbb{Z}}(M) \ar[r]^{F} \ar[d] & \End_{kQ}(M_0) \ar[d] \\
	\underline{\End}_{\Pi_w}^{\mathbb{Z}}(M) \ar[r]^(0.4){\underline{F}} & \End_{kQ}(M_0)/[T],
	}
\]
where $[T]$ is an ideal of $\End_{kQ}(M_0)$ consisting of morphisms factoring through objects in $\add T$, 
and vertical morphisms are canonical surjections.
\end{theorem}
\begin{proof}
In Proposition \ref{induced}, we show that $F$ actually induces a morphism $\underline{F}$.
$\underline{F}$ is surjective by Proposition \ref{propsur}.
In Proposition \ref{prop-uF-inj}, we show that $\underline{F}$ is injective.
\end{proof}
In Subsection \ref{subendalg}, we show one theorem which we will use to prove Proposition \ref{propsur}.
\begin{example}
Let $Q$ be a quiver
	\begin{xy} (0,0)="O",
	"O"+<0cm,0.35cm>="11"*{1},
	"11"+<-0.8cm,-0.7cm>="21"*{2},
	"21"+/r1.6cm/="22"*{3},
	
	\ar"11"+/dl/;"21"+/u/
	\ar"21"+/r/;"22"+/l/
	\ar"11"+/dr/;"22"+/u/
	\end{xy}.
Let $w=s_1s_2s_3s_1s_2s_1$.
This is a $c$-sortable element.
In Example \ref{exampleA2tilting}, we have
\begin{align*}
M&=M^{1}\oplus M^{2}\oplus M^{3}\oplus M^{4}\oplus M^{5}\oplus M^{6} \\
&={\bf 1}
\oplus 
	\begin{xy} (0,0)="O",
	"O"+<0.3cm,0.25cm>="11"*{{\bf 2}},
	"11"+<-0.3cm,-0.5cm>="21",
	"21"+/r0.6cm/="22"*{{\bf 1}},
	\ar@{-}"11"+/dr/;"22"+/u/<-2pt>
	\end{xy}
\oplus \begin{xy} (0,0)="O",
	"O"+<0cm,0.5cm>="11"*{{\bf 3}},
	"11"+<-0.3cm,-0.5cm>="21"*{{\bf 1}},
	"21"+<-0.3cm,-0.5cm>="31",
	"21"+/r0.6cm/="22"*{{\bf 2}},
	"31"+/r0.6cm/="32",
	"32"+/r0.6cm/="33"*{{\bf 1}},
	\ar@{-}"11"+/dl/;"21"+/u/<2pt>
	\ar@{-}"11"+/dr/;"22"+/u/<-2pt> 
	\ar@{-}"22"+/dr/;"33"+/u/<-2pt>
	\end{xy}
\oplus \left(
	\begin{xy} (0,0)="O",
	"O"+<0cm,0.7cm>="11"*{1},
	"11"+<-0.3cm,-0.5cm>="21"*{\bf 2},
	"21"+<-0.3cm,-0.5cm>="31",
	"31"+<-0.3cm,-0.5cm>="41",
	"21"+/r0.6cm/="22"*{\bf 3},
	"31"+/r0.6cm/="32"*{\bf 1},
	"32"+/r0.6cm/="33"*{\bf 2},
	"41"+/r0.6cm/="42",
	"42"+/r0.6cm/="43",
	"43"+/r0.6cm/="44"*{\bf 1},
	\ar@{-}"21"+/dr/;"32"+/u/<-2pt>
	\ar@{-}"22"+/dl/;"32"+/u/<2pt>
	\ar@{-}"22"+/dr/;"33"+/u/<-2pt> 
	\ar@{-}"33"+/dr/;"44"+/u/<-2pt> 
	\end{xy}
	\right)
\oplus \left(
	\begin{xy} (0,0)="O",
	"O"+<0cm,1cm>="11"*{2},
	"11"+<-0.3cm,-0.5cm>="21"*{\bf 3},
	"21"+<-0.3cm,-0.5cm>="31"*{\bf 1},
	"31"+<-0.3cm,-0.5cm>="41",
	"41"+<-0.3cm,-0.5cm>="51",
	"21"+/r0.6cm/="22"*{1},
	"31"+/r0.6cm/="32"*{\bf 2},
	"32"+/r0.6cm/="33"*{\bf 3},
	"41"+/r0.6cm/="42",
	"42"+/r0.6cm/="43"*{\bf 1},
	"43"+/r0.6cm/="44"*{\bf 2},
	"51"+/r0.6cm/="52",
	"52"+/r0.6cm/="53",
	"53"+/r0.6cm/="54",
	"54"+/r0.6cm/="55"*{\bf 1},
	\ar@{-}"11"+/dr/;"22"+/u/<-2pt> 
	\ar@{-}"21"+/dl/;"31"+/u/<2pt>
	\ar@{-}"21"+/dr/;"32"+/u/<-2pt>
	\ar@{-}"32"+/dr/;"43"+/u/<-2pt>
	\ar@{-}"33"+/dl/;"43"+/u/<2pt>
	\ar@{-}"33"+/dr/;"44"+/u/<-2pt>
	\ar@{-}"44"+/dr/;"55"+/u/<-2pt>
	\end{xy}\right)
\oplus \left(
	\begin{xy} (0,0)="O",
	"O"+<0cm,0.7cm>="11"*{1},
	"11"+<-0.3cm,-0.5cm>="21"*{2},
	"21"+<-0.3cm,-0.5cm>="31"*{\bf 3},
	"31"+<-0.3cm,-0.5cm>="41"*{\bf 1},
	"21"+/r0.6cm/="22"*{3},
	"31"+/r0.6cm/="32"*{1},
	"32"+/r0.6cm/="33"*{2},
	"41"+/r0.6cm/="42",
	"42"+/r0.6cm/="43",
	"43"+/r0.6cm/="44"*{1},
	\ar@{-}"21"+/dr/;"32"+/u/<-2pt>
	\ar@{-}"22"+/dl/;"32"+/u/<2pt>
	\ar@{-}"22"+/dr/;"33"+/u/<-2pt> 
	\ar@{-}"31"+/dl/;"41"+/u/<2pt>
	\ar@{-}"33"+/dr/;"44"+/u/<-2pt> 
	\end{xy}\right)
\end{align*}
in $\Sub^{\mathbb{Z}}\Pi_{w}$, where the degree zero parts are denoted by bold numbers.
Therefore, we have $P=M^{3}\oplus M^{5}\oplus M^{6}$ and
\begin{align*}
T=P_{0}=M^{3}_{0}\oplus M^{5}_{0}\oplus M^{6}_{0}
	=\begin{xy} (0,0)="O",
	"O"+<0cm,0.5cm>="11"*{3},
	"11"+<-0.3cm,-0.5cm>="21"*{1},
	"21"+<-0.3cm,-0.5cm>="31",
	"21"+/r0.6cm/="22"*{2},
	"31"+/r0.6cm/="32",
	"32"+/r0.6cm/="33"*{1},
	\ar@{-}"11"+/dl/;"21"+/u/<2pt>
	\ar@{-}"11"+/dr/;"22"+/u/<-2pt> 
	\ar@{-}"22"+/dr/;"33"+/u/<-2pt>
	\end{xy}
\oplus
	\begin{xy} (0,0)="O",
	"O"+<0cm,1.25cm>="11",
	"11"+<-0.3cm,-0.5cm>="21"*{3},
	"21"+<-0.3cm,-0.5cm>="31"*{1},
	"31"+<-0.3cm,-0.5cm>="41",
	"41"+<-0.3cm,-0.5cm>="51",
	"21"+/r0.6cm/="22",
	"31"+/r0.6cm/="32"*{2},
	"32"+/r0.6cm/="33"*{3},
	"41"+/r0.6cm/="42",
	"42"+/r0.6cm/="43"*{1},
	"43"+/r0.6cm/="44"*{2},
	"51"+/r0.6cm/="52",
	"52"+/r0.6cm/="53",
	"53"+/r0.6cm/="54",
	"54"+/r0.6cm/="55"*{1},
	"33"+/r0.6cm/="",
	\ar@{-}"21"+/dl/;"31"+/u/<2pt>
	\ar@{-}"21"+/dr/;"32"+/u/<-2pt>
	\ar@{-}"32"+/dr/;"43"+/u/<-2pt>
	\ar@{-}"33"+/dl/;"43"+/u/<2pt>
	\ar@{-}"33"+/dr/;"44"+/u/<-2pt>
	\ar@{-}"44"+/dr/;"55"+/u/<-2pt>
	\end{xy}
\oplus
	\begin{xy} (0,0)="00",
	"00"+<0.3cm,0.25cm>="11"*{3},
	"11"+<-0.3cm,-0.5cm>="21"*{1},
	\ar@{-}"11"+/dl/;"21"+/u/
	\end{xy},
\end{align*}
\begin{align*}
M_{0}=M^{1}_{0}\oplus M^{2}_{0}\oplus M^{4}_{0} \oplus T
	=1
\oplus
	\begin{xy} (0,0)="O",
	"O"+<0.3cm,0.25cm>="11"*{2},
	"11"+<-0.3cm,-0.5cm>="21",
	"21"+/r0.6cm/="22"*{1},
	\ar@{-}"11"+/dr/;"22"+/u/<-2pt>
	\end{xy}
\oplus
	\begin{xy} (0,0)="O",
	"O"+<0cm,1cm>="11",
	"11"+<-0.3cm,-0.5cm>="21"*{2},
	"21"+<-0.3cm,-0.5cm>="31",
	"31"+<-0.3cm,-0.5cm>="41",
	"21"+/r0.6cm/="22"*{3},
	"31"+/r0.6cm/="32"*{1},
	"32"+/r0.6cm/="33"*{2},
	"41"+/r0.6cm/="42",
	"42"+/r0.6cm/="43",
	"43"+/r0.6cm/="44"*{1},
	\ar@{-}"21"+/dr/;"32"+/u/<-2pt>
	\ar@{-}"22"+/dl/;"32"+/u/<2pt>
	\ar@{-}"22"+/dr/;"33"+/u/<-2pt> 
	\ar@{-}"33"+/dr/;"44"+/u/<-2pt> 
	\end{xy}
\oplus
	T.
\end{align*}
It is easy to see that the algebra $\End_{kQ}(M_{0})/[T]$ is given by the following quiver with relations
\[
\Delta=\left[\xymatrix{ \bullet \ar[r]^{a} & \bullet \ar[r]^{b} & \bullet}\right], \quad ab=0.
\]
By Theorem \ref{endalg} or a direct calculation, we can see that the algebra $\underline{\End}_{\Pi_{w}}^{\mathbb{Z}}(M)$ is also given by the same quiver with relations.

We can describe the Auslander-Reiten quiver of $\underline{\Sub}^{\mathbb{Z}}\Pi_w$.
Let $K$ be the kernel of the canonical epimorphism $\Pi_{w}e_{2} \to S_{2}$, where $S_{2}$ is a simple module associated with the vertex $2$, and let $N$ be the cokernel of an inclusion $(\Pi_{w}e_{1})_{1} \to \Pi_{w}e_{2}$:
\[
K=
	\begin{xy} (0,0)="O",
	"O"+<0cm,1cm>="11",
	"11"+<-0.3cm,-0.5cm>="21"*{3},
	"21"+<-0.3cm,-0.5cm>="31"*{1},
	"31"+<-0.3cm,-0.5cm>="41",
	"41"+<-0.3cm,-0.5cm>="51",
	"21"+/r0.6cm/="22"*{{\bf1}},
	"31"+/r0.6cm/="32"*{2},
	"32"+/r0.6cm/="33"*{3},
	"41"+/r0.6cm/="42",
	"42"+/r0.6cm/="43"*{1},
	"43"+/r0.6cm/="44"*{2},
	"51"+/r0.6cm/="52",
	"52"+/r0.6cm/="53",
	"53"+/r0.6cm/="54",
	"54"+/r0.6cm/="55"*{1},
	
	\ar@{-}"21"+/dl/;"31"+/u/<2pt>
	\ar@{-}"21"+/dr/;"32"+/u/<-2pt>
	\ar@{-}"32"+/dr/;"43"+/u/<-2pt>
	\ar@{-}"33"+/dl/;"43"+/u/<2pt>
	\ar@{-}"33"+/dr/;"44"+/u/<-2pt>
	\ar@{-}"44"+/dr/;"55"+/u/<-2pt>
	\end{xy}
,\qquad
N=
	\begin{xy} (0,0)="O",
	"O"+<0cm,0.5cm>="11"*{2},
	"11"+<-0.3cm,-0.5cm>="21"*{3},
	"21"+<-0.3cm,-0.5cm>="31"*{1},
	"31"+<-0.3cm,-0.5cm>="41",
	"41"+<-0.3cm,-0.5cm>="51",
	"21"+/r0.6cm/="22"*{1},
	
	\ar@{-}"11"+/dr/;"22"+/u/<-2pt>
	\ar@{-}"21"+/dl/;"31"+/u/<2pt>
	\end{xy}.
\]
Then the Auslander-Reiten quiver of $\underline{\Sub}^{\mathbb{Z}}\Pi_w$ is the following one:
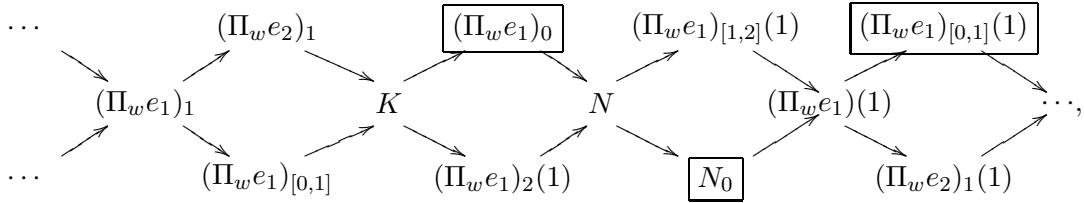
\[
\begin{xy} (0,0)="O",
	"O"+<1cm,0cm>="00",
	"00"+<3.2cm,0cm>="10"*{(\Pi_{w}e_{1})_{1}},
	"10"+<3.2cm,0cm>="20"*{K},
	"20"+<2.8cm,0cm>="30"*{N},
	"00"+<1.6cm,1cm>="11"*{\cdots},
	"00"+<1.6cm,-1cm>="1-1"*{\cdots},
	"10"+<1.6cm,1cm>="21"*{(\Pi_{w}e_{2})_{1}},
	"10"+<1.6cm,-1cm>="2-1"*{(\Pi_{w}e_{1})_{[0,1]}},
	"20"+<1.5cm,1cm>="31"*+[F]{(\Pi_{w}e_{1})_{0}},
	"20"+<1.5cm,-1cm>="3-1"*{(\Pi_{w}e_{1})_{2}(1)},
	"30"+<1.5cm,1cm>="41"*{(\Pi_{w}e_{1})_{[1,2]}(1)},
	"30"+<1.5cm,-1cm>="4-1"*+[F]{N_{0}},
	"30"+<3cm,0cm>="40"*{(\Pi_{w}e_{1})(1)},
	"41"+<3cm,0cm>="51"*+[F]{(\Pi_{w}e_{1})_{[0,1]}(1)},
	"4-1"+<3cm,0cm>="5-1"*{(\Pi_{w}e_{2})_{1}(1)},
	"40"+<3cm,0cm>="50"*{\cdots},
	
	
	\ar"11"+<0.5cm,-0.3cm>;"10"+<-0.5cm,0.3cm>
	\ar"1-1"+<0.5cm,0.3cm>;"10"+<-0.5cm,-0.3cm>
	
	\ar"10"+<0.5cm,0.3cm>;"21"+<-0.5cm,-0.3cm>
	\ar"10"+<0.5cm,-0.3cm>;"2-1"+<-0.5cm,0.3cm>
	
	\ar"21"+<0.5cm,-0.3cm>;"20"+<-0.2cm,0.3cm>
	\ar"2-1"+<0.5cm,0.3cm>;"20"+<-0.2cm,-0.3cm>
	
	\ar"20"+<0.2cm,0.3cm>;"31"+<-0.5cm,-0.3cm>
	\ar"20"+<0.2cm,-0.3cm>;"3-1"+<-0.5cm,0.3cm>	
	
	\ar"31"+<0.5cm,-0.3cm>;"30"+<-0.2cm,0.3cm>
	\ar"3-1"+<0.5cm,0.3cm>;"30"+<-0.2cm,-0.3cm>
	
	\ar"30"+<0.2cm,0.3cm>;"41"+<-0.5cm,-0.3cm>
	\ar"30"+<0.2cm,-0.3cm>;"4-1"+<-0.5cm,0.3cm>
	
	\ar"41"+<0.5cm,-0.3cm>;"40"+/ul/
	\ar"4-1"+<0.5cm,0.3cm>;"40"+/dl/
		
	\ar"40"+<0.2cm,0.3cm>;"51"+<-0.5cm,-0.3cm>
	\ar"40"+<0.2cm,-0.3cm>;"5-1"+<-0.5cm,0.3cm>
	
	\ar"51"+<0.5cm,-0.3cm>;"50"+/ul/
	\ar"5-1"+<0.5cm,0.3cm>;"50"+/dl/
\end{xy}
,
\]
where $M=(\Pi_{w}e_{1})_{0}\oplus N_{0}\oplus (\Pi_{w}e_{1})_{[0,1]}(1)$.
We see that the shape of the Auslander-Reiten quiver of $\underline{\Sub}^{\mathbb{Z}}\Pi_{w}$ is actually the same as that of $\mathsf{D}^{\rm b}(\underline{\End}_{\Pi_{w}}^{\mathbb{Z}}(M))$.
\end{example}
We first describe the quiver of $\End_{\Pi_w}(M)$.
We recall the following definition of a quiver $Q_{w}$ associated with a reduced expression $w=s_{u_1}s_{u_2}\cdots s_{u_l}$. 
This $Q_{w}$ was denoted by $Q(u_{1},\ldots,u_{l})$ in \cite[Subsection I\hspace{-.1em}I\hspace{-.1em}I. 4]{BIRSc}.
\begin{definition}\cite{BIRSc}
We define a quiver $Q_{w}$ associated with a reduced expression $w=s_{u_1}s_{u_2}\cdots s_{u_l}$ as follows:
\begin{itemize}
	\item vertices: $(Q_{w})_{0}=\{ 1, 2, \ldots, l \}$.
	
	A vertex $1\leq i \leq l$ in $Q_{w}$ is said to be {\it type $u \in Q_{0}$} if $u_{i}=u$.
	\item arrows: 
		\begin{itemize}
		
		\item[(a1)]
		For each $u\in \supp(w)$, draw an arrow from $j$ to $i$, where $i,j$ are vertices of type $u$, $i<j$, and there is no vertex of type $u$ between $i$ and $j$ {\rm(}we call these arrows {\it going to the left} {\rm)}.
				
		\item[(a2)]
		For each arrow $\alpha : u \to v \in Q_{1}$, draw an arrow $\alpha_{i}$ from $i$ to $j$, where $i<j$, $i$ is a vertex of type $u$, $j$ is a vertex of type $v$, there is no vertex of type $u$ between $i$ and $j$, and $j$ is the biggest vertex of type $v$ before the next vertex of type $u$ (we call these arrows {\it $Q$-arrows}).
		\item[(a3)]
		For each arrow $\alpha : u \to v \in Q_{1}$, draw an arrow $\alpha_{i}^{\ast}$ from $i$ to $j$, where $i<j$, $i$ is a vertex of type $v$, $j$ is a vertex of type $u$, there is no vertex of type $v$ between $i$ and $j$, and $j$ is the biggest vertex of type $u$ before the next vertex of type $v$ (we call these arrows {\it $Q^{\ast}$-arrows}).
		\end{itemize}
	\end{itemize}
\end{definition}
Note that the quiver $Q_{w}$ depends on the choice of a reduced expression of $w$.
\begin{example}\label{exquiver}
(a) Let $Q$ be the quiver 
	\begin{xy} (0,0)+<0cm,0.2cm>="O",
	"O"+<0cm,0.35cm>="11"*{1},
	"11"+<-0.6cm,-0.7cm>="21"*{2},
	"21"+/r1.2cm/="22"*{3},
	
	\ar_{\alpha}"11"+/dl/;"21"+/u/
	\ar_{\beta}"21"+/r/;"22"+/l/
	\ar^{\gamma}"11"+/dr/;"22"+/u/
	\end{xy}, 
and $w=s_{u_{1}}s_{u_{2}}s_{u_{3}}s_{u_{4}}s_{u_{5}}s_{u_{6}}=s_{1}s_{2}s_{3}s_{1}s_{2}s_{1}$.
Then we have the quiver $Q_{w}$ as follows:
$$
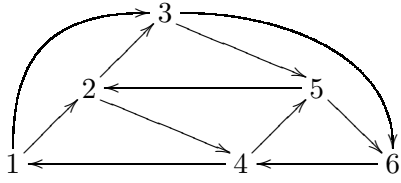
\begin{xy}
	(0,0)+<0cm,0.2cm>="O",
	"O"+<0cm,1cm>="3"*{3},
	"O"+<-1cm,0cm>="2"*{2},
	"O"+<-2cm,-1cm>="1"*{1},
	"1"+<3cm,0cm>="4"*{4},
	"2"+<3cm,0cm>="5"*{5},
	"4"+<2cm,0cm>="6"*{6},
	
	\ar"1"+/ur/;"2"+/dl/
	\ar"2"+/ur/;"3"+/dl/
	\ar@(u,l)"1"+/u/;"3"+/l/
	\ar"2"+/dr/;"4"+/ul/
	\ar"3"+/dr/;"5"+/ul/
	\ar"4"+/ur/;"5"+/dl/
	\ar"5"+/dr/;"6"+/ul/
	\ar@(r,u)"3"+/r/;"6"+/u/
	
	\ar"4"+/l/;"1"+/r/
	\ar"5"+/l/;"2"+/r/
	\ar"6"+/l/;"4"+/r/
\end{xy}
$$

(b) Let $Q$ be the same quiver in (a), and $w^{\prime}=s_{u_{1}}s_{u_{2}}s_{u_{3}}s_{u_{4}}s_{u_{5}}s_{u_{6}}=s_{1}s_{2}s_{3}s_{2}s_{1}s_{2}$. This is another reduced expression of $w$ in (a).
Then we have the quiver $Q_{w^{\prime}}$ as follows:
$$
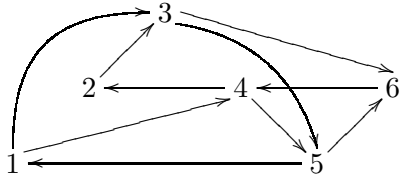
\begin{xy}
	(0,0)+<0cm,0.2cm>="O",
	"O"+<0cm,1cm>="3"*{3},
	"O"+<-1cm,0cm>="2"*{2},
	"O"+<-2cm,-1cm>="1"*{1},
	"2"+<2cm,0cm>="4"*{4},
	"1"+<4cm,0cm>="5"*{5},
	"4"+<2cm,0cm>="6"*{6},
	
	\ar"2"+/ur/;"3"+/dl/
	\ar@(u,l)"1"+/u/;"3"+/l/
	\ar"1"+/ur/;"4"+/dl/
	\ar@/^0.4cm/"3"+/dr/;"5"+/u/
	\ar"4"+/dr/;"5"+/ul/
	\ar"5"+/ur/;"6"+/dl/
	\ar"3"+/r/;"6"+/u/
	
	\ar"4"+/l/;"2"+/r/
	\ar"6"+/l/;"4"+/r/
	\ar"5"+/l/;"1"+/r/
\end{xy}
$$
\end{example}
It is shown that $Q_{w}$ gives a quiver of $\End_{\Pi_{w}}(M)$ as we see in Theorem \ref{quiver}.
We define a morphism of  algebras $\phi : kQ_{w}\to \End_{\Pi_{w}}(M)$ by
\begin{itemize}
	\item[(a1)] For an arrow $\beta : j \to i$ going to the left, $\phi(\beta)$ is the canonical surjection $M^{j} \to M^{i}$.
	\item[(a2)] For a $Q$-arrow $\alpha_{i}: i \to j$ of the arrow $\alpha\in Q_{1}$, $\phi(\alpha_{i})$ is a morphism of $\Pi_{w}$-modules from $M^{i}$ to $M^{j}$ given by multiplying $\alpha$ from the right.
	\item[(a3)] For a $Q^{\ast}$-arrow $\alpha_{i}^{\ast}: i \to j$ of the arrow $\alpha\in Q_{1}$, $\phi(\alpha_{i}^{\ast})$ is a morphism of $\Pi_{w}$-modules from $M^{i}$ to $M^{j}$ given by multiplying $\alpha^{\ast}$ from the right.
	\end{itemize}
In Theorem \ref{quiver}, we do not consider gradings of $\Pi_{w}$ and $M^{i}$.
\begin{theorem} \cite[Theorem I\hspace{-.1em}I\hspace{-.1em}I. 4.1]{BIRSc} \label{quiver}
Let $w=s_{u_1}s_{u_2}\cdots s_{u_l}$ be a reduced expression.
Then the morphism of  algebras $\phi : kQ_{w}\to \End_{\Pi_{w}}(M)$ induces an isomorphism of algebras \[ \underline{\phi} : kQ_{w}/I \simeq \End_{\Pi_{w}}(M) \] for an ideal $I$ of $kQ$.
\end{theorem}
Since $\End_{\Pi_{w}}(M)=\bigoplus_{n\in\mathbb{Z}}\Hom_{\Pi_{w}}^{\mathbb{Z}}(M,M(n))$,
we regard $\End_{\Pi_{w}}(M)$ as a graded algebra by $\End_{\Pi_{w}}(M)_{n}=\Hom_{\Pi_{w}}^{\mathbb{Z}}(M,M(n))$.
In particular, we have $\End_{\Pi_{w}}^{\mathbb{Z}}(M)=\End_{\Pi_{w}}(M)_{0}$.
We introduce a grading on $Q_{w}$.
\begin{definition} \label{gradeofquiver}
Assume that $w=s_{u_1}s_{u_2}\cdots s_{u_l}$ is a reduced expression.
Let $Q_{w}$ be the quiver of $\End_{\Pi_w}(M)$ and $(Q_{w})_{0}=\{1, \ldots, l\}$.
We define a grading on $Q_{w}$ as follows:
		\begin{itemize}
		\item[(1)]
		All arrows going to the left are of degree one.
		\item[(2)]
		Let $\beta : i \to j$ be a $Q$-arrow in $Q_{w}$. Then the degree of $\beta$ is $m_{i}-m_{j}$. 
		\item[(3)]
		Let $\beta : i \to j$ be a $Q^{\ast}$-arrow in $Q_{w}$. Then the degree of $\beta$ is $m_{i}-m_{j}+1$. 
		\end{itemize}
\end{definition}
\begin{example}
(a) In the quiver of Example \ref{exquiver} (a), we have the grading of $Q_{w}$ as follows:
$$
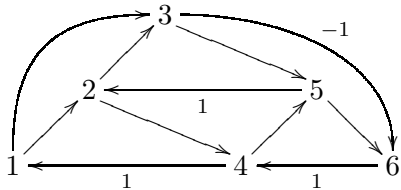
\begin{xy}
	(0,0)+<0cm,0.2cm>="O",
	"O"+<0cm,1cm>="3"*{3},
	"O"+<-1cm,0cm>="2"*{2},
	"O"+<-2cm,-1cm>="1"*{1},
	"1"+<3cm,0cm>="4"*{4},
	"2"+<3cm,0cm>="5"*{5},
	"4"+<2cm,0cm>="6"*{6},
	
	\ar"1"+/ur/;"2"+/dl/
	\ar"2"+/ur/;"3"+/dl/
	\ar@(u,l)"1"+/u/;"3"+/l/
	\ar"2"+/dr/;"4"+/ul/
	\ar"3"+/dr/;"5"+/ul/
	\ar"4"+/ur/;"5"+/dl/
	\ar"5"+/dr/;"6"+/ul/
	\ar@(r,u)^{-1}"3"+/r/;"6"+/u/
	
	\ar^{1}"4"+/l/;"1"+/r/
	\ar^{1}"5"+/l/;"2"+/r/
	\ar^{1}"6"+/l/;"4"+/r/
\end{xy}
$$
where non numbered arrows have degree zero.

(b) In the quiver of Example \ref{exquiver} (b), we have the grading of $Q_{w}$ as follows:
$$
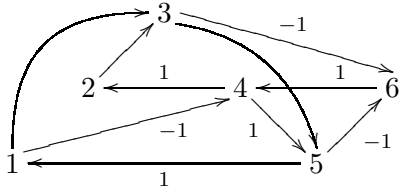
\begin{xy}
	(0,0)+<0cm,0.2cm>="O",
	"O"+<0cm,1cm>="3"*{3},
	"O"+<-1cm,0cm>="2"*{2},
	"O"+<-2cm,-1cm>="1"*{1},
	"2"+<2cm,0cm>="4"*{4},
	"1"+<4cm,0cm>="5"*{5},
	"4"+<2cm,0cm>="6"*{6},
	
	\ar"2"+/ur/;"3"+/dl/
	\ar@(u,l)"1"+/u/;"3"+/l/
	\ar_(0.7){-1}"1"+/ur/;"4"+/dl/
	\ar@/^0.4cm/"3"+/dr/;"5"+/u/
	\ar_(0.3){1}"4"+/dr/;"5"+/ul/
	\ar_{-1}"5"+/ur/;"6"+/dl/
	\ar^{-1}"3"+/r/;"6"+/u/
	
	\ar_{1}"4"+/l/;"2"+/r/
	\ar_(0.3){1}"6"+/l/;"4"+/r/
	\ar^{1}"5"+/l/;"1"+/r/
\end{xy}
$$
where non numbered arrows have degree zero.
\end{example}
We regard $kQ_{w}$ as a graded algebra by the grading of Definition \ref{gradeofquiver}.
Then the isomorphism in Theorem \ref{quiver} holds as graded algebras.
\begin{proposition}\label{gradedhom}
The morphism of algebras \[\phi : kQ_{w} \to \End_{\Pi_{w}}(M) \] is a surjective morphism of graded algebras.
\end{proposition}
\begin{proof}
It is enough to show that the morphism $\phi : kQ_{w} \to \End_{\Pi_{w}}(M)$ preserves gradings.
Since $kQ_{w}$ is generated by arrows, it is enough to show that $\phi$ preserves gradings of arrows.

(a1)
Let $\beta : j \to i$ be an arrow going to the left.
Then $\phi(\beta)$ is given by a surjection \[(\Pi/I_{u_1u_2\cdots u_j})e_{u_{j}}(m_{j}) \to (\Pi/I_{u_1u_2\cdots u_i})e_{u_{i}}(m_{i}).\]
Since there exists no vertex of type $u_{i}=u_{j}$ between $i$ and $j$, we have $m_{i}+1=m_{j}$.
Since $\top\left( M^{j} \right)$ is concentrated in $-m_{j}$ and $\top \left( M^{i} \right)$ is concentrated in $-m_{i}$, 
this surjection is degree one.

(a2)
Let $\beta=\alpha_{i} : i \to j$ be a $Q$-arrow in $Q_{w}$, where $\alpha\in Q_{1}$.
Then $\phi(\beta)$ is a morphism multiplying $\alpha$ from the right:
\[\phi(\beta)=(\cdot \alpha) : (\Pi/I_{u_1u_2\cdots u_i})e_{u_{i}}(m_{i}) \to (\Pi/I_{u_1u_2\cdots u_j})e_{u_{j}}(m_{j}).\]
This means $\phi(\beta)$ is degree $m_{i}-m_{j}$.

(a3)
Let $\beta=\alpha_{i}^{\ast} : i \to j$ be a $Q^{\ast}$-arrow in $Q_{w}$, where $\alpha\in Q_{1}$.
Then $\phi(\beta)$ is a morphism multiplying $\alpha^{\ast}$ from the right:
\[\phi(\beta)=(\cdot \alpha^{\ast}) : (\Pi/I_{u_1u_2\cdots u_i})e_{u_{i}}(m_{i}) \to (\Pi/I_{u_1u_2\cdots u_j})e_{u_{j}}(m_{j}).\]
This means $\phi(\beta)$ is degree $m_{i}-m_{j}+1$.
\end{proof}
The following lemma is important to show Propositions \ref{prop-gldim-fin} and \ref{induced}.
\begin{lemma}\label{arrownega}
Assume that $w=s_{u_1}s_{u_2}\cdots s_{u_l}$ is a $c$-sortable element.
Let $\beta : i \to j$ be an arrow in $Q_{w}$ which is a $Q$-arrow or a $Q^{\ast}$-arrow.
Then the following holds.
\begin{itemize}
\item[(a)]
If $\beta$ has a negative degree, then we have $i=p_{u_{i}}$ and $j=p_{u_{j}}$.

\item[(b)]
If $\beta$ satisfies $i\not=p_{u_{i}}$ or $j\not=p_{u_{j}}$, then $\beta$ has degree zero.
\end{itemize}
\end{lemma}
\begin{proof}
Assume that $\beta$ is a $Q$-arrow and $i$ is a vertex of type $u$ and $j$ is a vertex of type $v$.
Then, by the definition of $Q_{w}$, there exists an arrow $\alpha: u \to v$ in $Q$ which satisfies $\alpha_{i}=\beta$.
Pick up vertices of type $u$ and $v$ from $(Q_{w})_{0}=\{1,2,\ldots,l\}$, then we have the following two cases:
\begin{align} 1\leq a_{1}< b_{1} < a_{2}< b_{2} < \cdots < a_{s} < b_{s} < b_{s+1} < \cdots < b_{t}\leq l, \label{abb} \\ 
1\leq a_{1}< b_{1} < a_{2}< b_{2} < \cdots < b_{t} < a_{t+1} < a_{t+2} < \cdots < a_{s}\leq l,\label{abaa} \end{align}
where $a_{\bullet}$ are vertices of type $u$ and $b_{\bullet}$ are vertices of type $v$.
By the definition of $m_{j}$, we have $m_{a_{k}}=k-1$ and $m_{b_{k}}=k-1$.
Moreover, by the definition of $p_{u}$ and $p_{v}$, we have $p_{u}=a_{s}$ and $p_{v}=b_{t}$.
Let $Q_{w}(i,j,\alpha)$ be a subquiver of $Q_{w}$ such that $\left( Q_{w}(i,j,\alpha) \right)_{0}=\{ a_{1}, \cdots, a_{s}, b_{1}, \cdots, b_{t} \}$ and $\left( Q_{w}(i,j,\alpha) \right)_{1}$ is the set of all arrows of the form $\alpha_{k}$ or $\alpha_{k}^{\ast}$ for some $1\leq k\leq l$ or arrows going to the left.

In the case (\ref{abb}), the quiver $Q_{w}(i,j,\alpha)$ is the following:
$$
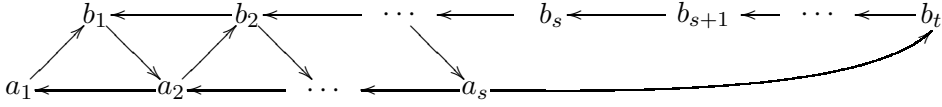
\begin{xy}
	(0,0)+<0cm,0.2cm>="O",
	"O"+<0cm,-0.5cm>="a1"*{a_{1}},
	"O"+<1cm,0.5cm>="b1"*{b_{1}},
	"a1"+<2cm,0cm>="a2"*{a_{2}},
	"b1"+<2cm,0cm>="b2"*{b_{2}},
	"a2"+<4cm,0cm>="as"*{a_{s}},
	"b2"+<4cm,0cm>="bs"*{b_{s}},
	"bs"+<2cm,0cm>="bs+1"*{b_{s+1}},
	"bs+1"+<3cm,0cm>="bt"*{b_{t}},
	"b2"+<2cm,0cm>="b3"*{\cdots},
	"a2"+<2cm,0cm>="a3"*{\cdots},
	"bs+1"+<1.5cm,0cm>="bs+2"*{\cdots},
	
	\ar"a1"+/ur/;"b1"+/dl/
	\ar"b1"+/dr/;"a2"+/ul/
	\ar"a2"+/ur/;"b2"+/dl/
	\ar@(r,dl)"as"+/r/;"bt"+/d/
	
	\ar"a2"+/l/;"a1"+/r/
	\ar"b2"+/l/;"b1"+/r/
	\ar"b2"+/dr/;"a3"+/ul/
	\ar"b3"+/dr/;"as"+/ul/
	\ar"as"+/l/;"a3"+<0.5cm,0cm>
	\ar"a3"+<-0.5cm,0cm>;"a2"+/r/
	\ar"b3"+<-0.5cm,0cm>;"b2"+/r/
	\ar"bs+1"+<-0.5cm,0cm>;"bs"+/r/
	\ar"bt"+/l/;"bs+2"+<0.5cm,0cm>
	\ar"bs+2"+<-0.5cm,0cm>;"bs+1"+<0.5cm,0cm>
	\ar"bs"+<-0.5cm,0cm>;"b3"+<0.5cm,0cm>
\end{xy}.$$
Since $\beta$ is a $Q$-arrow, $\beta$ is one of the arrows of $a_{k}\to b_{k}$ for $1\leq k \leq s-1$ or $a_{s}\to b_{t}$.
For $1\leq k \leq s-1$, we have $m_{a_{k}}-m_{b_{k}}=0$.
Therefore, in the case (\ref{abb}), $({\rm a})$ and $({\rm b})$ hold.

In the case (\ref{abaa}), the quiver $Q_{w}(i,j,\alpha)$ is the following:
$$
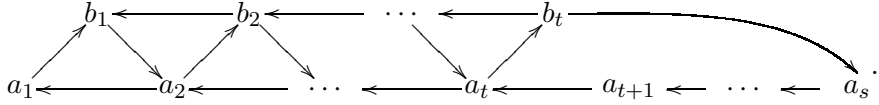
\begin{xy}
	(0,0)+<0cm,0.2cm>="O",
	"O"+<0cm,-0.5cm>="a1"*{a_{1}},
	"O"+<1cm,0.5cm>="b1"*{b_{1}},
	"a1"+<2cm,0cm>="a2"*{a_{2}},
	"b1"+<2cm,0cm>="b2"*{b_{2}},
	"a2"+<4cm,0cm>="at"*{a_{t}},
	"b2"+<4cm,0cm>="bt"*{b_{t}},
	"at"+<2cm,0cm>="at+1"*{a_{t+1}},
	"at+1"+<3cm,0cm>="as"*{a_{s}},
	"b2"+<2cm,0cm>="b3"*{\cdots},
	"a2"+<2cm,0cm>="a3"*{\cdots},
	"at+1"+<1.5cm,0cm>="at+2"*{\cdots},
	
	\ar"a1"+/ur/;"b1"+/dl/
	\ar"b1"+/dr/;"a2"+/ul/
	\ar"a2"+/ur/;"b2"+/dl/
	\ar"at"+/ur/;"bt"+/dl/
	\ar@(r,ul)"bt"+/r/;"as"+/u/
	
	\ar"a2"+/l/;"a1"+/r/
	\ar"b2"+/l/;"b1"+/r/
	\ar"b2"+/dr/;"a3"+/ul/
	\ar"b3"+/dr/;"at"+/ul/
	\ar"at"+/l/;"a3"+<0.5cm,0cm>
	\ar"a3"+<-0.5cm,0cm>;"a2"+/r/
	\ar"b3"+<-0.5cm,0cm>;"b2"+/r/
	\ar"at+1"+<-0.5cm,0cm>;"at"+/r/
	\ar"bt"+/l/;"b3"+<0.5cm,0cm>
	\ar"at+2"+<-0.5cm,0cm>;"at+1"+<0.5cm,0cm>
	\ar"as"+<-0.5cm,0cm>;"at+2"+<0.5cm,0cm>
\end{xy}.$$
Since $\beta$ is a $Q$-arrow, $\beta$ is one of the arrows of $a_{k}\to b_{k}$ for $1\leq k \leq t$.
For $1\leq k \leq t$, we have $m_{a_{k}}-m_{b_{k}}=0$.
Therefore, in the case (\ref{abaa}), $({\rm a})$ and $({\rm b})$ hold.

By the same argument, we can show in the case when $\beta$ is a $Q^{\ast}$-arrow.
\end{proof}
\begin{lemma}\label{facthr}
Assume that $w=s_{u_1}s_{u_2}\cdots s_{u_l}$ is a $c$-sortable element,
then any $f \in \Hom_{\Pi_{w}}^{\mathbb{Z}}(M,M(a))$ with $a<0$ factors through $\add P=\add (\bigoplus_{u \in \supp(w)}M^{p_{u}})$.
\end{lemma}
\begin{proof}
We identify $\End_{\Pi_{w}}(M)$ with $kQ_{w}/I$ as graded algebras by Theorem \ref{quiver} and Proposition \ref{gradedhom}.
Since $f$ is written as a liner combination of paths in $Q_{w}$, 
we can assume that $f=p$ for some path $p$ in $Q_{w}$.
Since $f$ has a negative degree, the degree of $p$ is negative.
Thus $p$ contains an arrow of negative degree.
By Lemma \ref{arrownega}, $p$ factors through a vertex $p_{u}$ for some $u\in\supp(w)$.
Therefore, $f$ factors through $\add P=\add (\bigoplus_{u \in \supp(w)}M^{p_{u}})$.
\end{proof}
Now we are ready to show the finiteness of the global dimension of $\underline{\End}_{\Pi_{w}}^{\mathbb{Z}}(M)$.
\begin{proposition}\label{prop-gldim-fin}
Let $w=s_{u_1}s_{u_2}\cdots s_{u_l}=c^{(0)}c^{(1)}\cdots c^{(m)}$ be a $c$-sortable element and $M=\bigoplus_{i=1}^{l}M^{i}$ be a tilting object in $\underline{\Sub}^{\mathbb{Z}}\Pi_{w}$.
Then the global dimension of $\underline{\End}_{\Pi_{w}}^{\mathbb{Z}}(M)$ is finite.
\end{proposition}
\begin{proof}
We denote by $\underline{Q}_{w}$ the full subquiver of $Q_{w}$ such that $(\underline{Q}_{w})_{0}=(Q_{w})_{0}\setminus\{p_{u}\mid u\in\supp(w)\}$.
Then $\phi$ induces a surjective morphism of graded algebras $\widetilde{\phi} : k\underline{Q}_{w}\to \underline{\End}_{\Pi_{w}}(M)$ by \cite[Theorem 6.6]{BIRSm}.
By Lemma \ref{arrownega} (a), $k\underline{Q}_{w}$ is positively graded and therefore $\underline{\End}_{\Pi_{w}}(M)$ is also positively graded.
By taking degree zero part of these algebras, we have the following commutative diagram
$$
\xymatrix{
k\underline{Q}_{w}\ar[r]^{\widetilde{\phi}}\ar[d] & \underline{\End}_{\Pi_{w}}(M)\ar[d]\\
k\underline{Q}_{w, 0}\ar[r]^{\overline{\phi}} & \underline{\End}_{\Pi_{w}}^{\mathbb{Z}}(M),
}
$$
where we denote by $\underline{Q}_{w, 0}$ a subquiver of $\underline{Q}_{w}$ such that vertices are same as $\underline{Q}_{w}$ and arrows are all degree zero arrows of $\underline{Q}_{w}$.
We have a surjection $\overline{\phi}$, since $\widetilde{\phi}$ and vertical morphisms are surjections.
Because $\underline{Q}_{w, 0}$ does not contains arrows going to the left, $\underline{Q}_{w, 0}$ is acyclic.
Therefore the global dimension of $\underline{\End}_{\Pi_{w}}^{\mathbb{Z}}(M)$ is finite.
\end{proof}
In Section \ref{gldimofend}, we show that the global dimension of $\underline{\End}_{\Pi_{w}}^{\mathbb{Z}}(M)$ is at most two.
We show that the morphism $F$ actually induces a morphism $\underline{F}$.
\begin{proposition}\label{induced}
The morphism $F$ induces a morphism of algebras:
\[ \underline{F} : \underline{\End}_{\Pi_w}^{\mathbb{Z}}(M) \to \End_{kQ}(M_0)/[T] .\]
\end{proposition}
\begin{proof}
We show that if a morphism $f:M \to M$ in $\mod^{\mathbb{Z}}\Pi_{w}$ factors through graded projective $\Pi_{w}$-modules, 
then $f$ factors through $\add P=\add (\bigoplus_{u \in \supp(w)}M^{p_{u}})$.
Without loss of generality, we may assume that $f=h\circ g$ for $g:M \to M^{p_{u}}(a)$ and $h:M^{p_{u}}(a) \to M$, where $u\in \supp(w)$ and $a \in \mathbb{Z}$.
We divide into three cases:
	\begin{itemize}
		\item If $a>0$, then $M^{p_{u}}(a)_{0}=M^{p_{u}}_{a}=0$, since $M^{p_{u}}=M^{p_{u}}_{\leq 0}$ by Proposition \ref{pii}. Thus we have $f|_{0}=0$.
		\item If $a=0$, then $f$ actually factors through $M^{p_{u}} \in \add P$.
		\item If $a<0$, then $g$ factors through $\add P$ by Lemma \ref{facthr}. Thus $f$ also factors through $\add P$.
		\end{itemize}
\vspace{-0.5cm}\end{proof}
In the rest of this subsection, we give some examples of tilting objects $M$ and its endomorphism algebras.
\begin{example}
If $Q$ is not Dynkin, then $w=c^2=s_{u_1}s_{u_2}\cdots s_{u_n}s_{u_1}s_{u_2}\cdots s_{u_n}$ is a reduced expression by \cite[Proposition I\hspace{-.1em}I\hspace{-.1em}I. 3.1]{BIRSc}.
Thus we have $\Pi_w=\Pi_{\leq 1}$ by Proposition \ref{pii}.
Since $M=\Pi_{c}\oplus\Pi_{c^{2}}(1)\simeq kQ$ in $\underline{\Sub}^{\mathbb{Z}}\Pi_w$ and $kQ$ is concentrated in degree $0$, we have
\[
\underline{\End}_{\Pi_w}^{\mathbb{Z}}(M)=\underline{\End}_{\Pi_w}(kQ).
\]
By \cite[Proposition I\hspace{-.1em}I\hspace{-.1em}I. 3.2]{BIRSc}, we have an isomorphism $\underline{\End}_{\Pi_w}(kQ)\simeq kQ$.
Therefore, we have $\underline{\End}_{\Pi_w}^{\mathbb{Z}}(M)\simeq kQ$, and a triangulated equivalence
\[
 \underline{\Sub}^{\mathbb{Z}}\Pi_w \simeq {\mathsf K}^{{\rm b}}(\proj kQ) \simeq {\mathsf D}^{{\rm b}}(kQ).
 \]
\end{example}
\begin{example}
Let $Q$ be a quiver \xymatrix{1 \ar[r]<1.5pt> \ar[r]<-1.5pt> & 2}.
Then we have a graded algebra $\Pi=\Pi e_1\oplus\Pi e_2$, and these are represented by their radical filtrations as follows:
\[
	\begin{xy} (0,0)="00",
	"00"+<0cm,1.6cm>="11"*{{\bf 1}},
	"11"+<-0.3cm,-0.5cm>="21"*{2},
	"21"+<-0.3cm,-0.5cm>="31"*{1},
	"31"+<-0.3cm,-0.5cm>="41"*{2},
	"41"+<-0.3cm,-0.5cm>="51"*{1},
	"51"+<-0.3cm,-0.5cm>="61"*{2},
	"61"+<-0.3cm,-0.5cm>="71",
	"21"+/r0.6cm/="22"*{2},
	"31"+/r0.6cm/="32"*{1},
	"32"+/r0.6cm/="33"*{1},
	"41"+/r0.6cm/="42"*{2},
	"42"+/r0.6cm/="43"*{2},
	"43"+/r0.6cm/="44"*{2},
	"51"+/r0.6cm/="52"*{1},
	"52"+/r0.6cm/="53"*{1},
	"53"+/r0.6cm/="54"*{1},
	"54"+/r0.6cm/="55"*{1},
	"61"+/r0.6cm/="62"*{2},
	"62"+/r0.6cm/="63"*{2},
	"63"+/r0.6cm/="64"*{2},
	"64"+/r0.6cm/="65"*{2},
	"65"+/r0.6cm/="66"*{2},
	"71"+/r0.6cm/="72",
	"72"+/r0.6cm/="73",
	"73"+/r0.6cm/="74",
	"74"+/r0.6cm/="75",
	"75"+/r0.6cm/="76",
	"76"+/r0.6cm/="77",
	
	\ar@{-}"21"+/dl/;"31"+/u/<2pt>
	\ar@{-}"21"+/dr/;"32"+/u/<-2pt>
	\ar@{-}"22"+/dl/;"32"+/u/<2pt>
	\ar@{-}"22"+/dr/;"33"+/u/<-2pt>
	\ar@{-}"41"+/dl/;"51"+/u/<2pt>
	\ar@{-}"41"+/dr/;"52"+/u/<-2pt>
	\ar@{-}"42"+/dl/;"52"+/u/<2pt>
	\ar@{-}"42"+/dr/;"53"+/u/<-2pt>
	\ar@{-}"43"+/dl/;"53"+/u/<2pt>
	\ar@{-}"43"+/dr/;"54"+/u/<-2pt>
	\ar@{-}"44"+/dl/;"54"+/u/<2pt>
	\ar@{-}"44"+/dr/;"55"+/u/<-2pt> 
	\ar@{-}"61"+/dl/;"71"+/u/<2pt>
	\ar@{-}"61"+/dr/;"72"+/u/<-2pt>
	\ar@{-}"62"+/dl/;"72"+/u/<2pt>
	\ar@{-}"62"+/dr/;"73"+/u/<-2pt>
	\ar@{-}"63"+/dl/;"73"+/u/<2pt>
	\ar@{-}"63"+/dr/;"74"+/u/<-2pt>
	\ar@{-}"64"+/dl/;"74"+/u/<2pt>
	\ar@{-}"64"+/dr/;"75"+/u/<-2pt>
	\ar@{-}"65"+/dl/;"75"+/u/<2pt>
	\ar@{-}"65"+/dr/;"76"+/u/<-2pt>
	\ar@{-}"66"+/dl/;"76"+/u/<2pt>
	\ar@{-}"66"+/dr/;"77"+/u/<-2pt>
	\end{xy}
\qquad
	\begin{xy} (0,0)="00",
	"00"+<0cm,1.6cm>="11"*{{\bf 2}},
	"11"+<-0.3cm,-0.5cm>="21"*{{\bf 1}},
	"21"+<-0.3cm,-0.5cm>="31"*{2},
	"31"+<-0.3cm,-0.5cm>="41"*{1},
	"41"+<-0.3cm,-0.5cm>="51"*{2},
	"51"+<-0.3cm,-0.5cm>="61"*{1},
	"61"+<-0.3cm,-0.5cm>="71",
	"21"+/r0.6cm/="22"*{{\bf 1}},
	"31"+/r0.6cm/="32"*{2},
	"32"+/r0.6cm/="33"*{2},
	"41"+/r0.6cm/="42"*{1},
	"42"+/r0.6cm/="43"*{1},
	"43"+/r0.6cm/="44"*{1},
	"51"+/r0.6cm/="52"*{2},
	"52"+/r0.6cm/="53"*{2},
	"53"+/r0.6cm/="54"*{2},
	"54"+/r0.6cm/="55"*{2},
	"61"+/r0.6cm/="62"*{1},
	"62"+/r0.6cm/="63"*{1},
	"63"+/r0.6cm/="64"*{1},
	"64"+/r0.6cm/="65"*{1},
	"65"+/r0.6cm/="66"*{1},
	"71"+/r0.6cm/="72",
	"72"+/r0.6cm/="73",
	"73"+/r0.6cm/="74",
	"74"+/r0.6cm/="75",
	"75"+/r0.6cm/="76",
	"76"+/r0.6cm/="77",
	
	\ar@{-}"11"+/dl/;"21"+/u/<2pt>
	\ar@{-}"11"+/dr/;"22"+/u/<-2pt>
	\ar@{-}"31"+/dl/;"41"+/u/<2pt>
	\ar@{-}"31"+/dr/;"42"+/u/<-2pt>
	\ar@{-}"32"+/dl/;"42"+/u/<2pt>
	\ar@{-}"32"+/dr/;"43"+/u/<-2pt>
	\ar@{-}"33"+/dl/;"43"+/u/<2pt>
	\ar@{-}"33"+/dr/;"44"+/u/<-2pt>
	\ar@{-}"51"+/dl/;"61"+/u/<2pt>
	\ar@{-}"51"+/dr/;"62"+/u/<-2pt>
	\ar@{-}"52"+/dl/;"62"+/u/<2pt>
	\ar@{-}"52"+/dr/;"63"+/u/<-2pt>
	\ar@{-}"53"+/dl/;"63"+/u/<2pt>
	\ar@{-}"53"+/dr/;"64"+/u/<-2pt>
	\ar@{-}"54"+/dl/;"64"+/u/<2pt>
	\ar@{-}"54"+/dr/;"65"+/u/<-2pt>
	\ar@{-}"55"+/dl/;"65"+/u/<2pt>
	\ar@{-}"55"+/dr/;"66"+/u/<-2pt>
	\end{xy},
\]
where the degree zero parts are denoted by bold numbers.
Let $c=s_{1}s_{2}$.
This is a Coxeter element.
Let $w=c^{n+1}=s_{1}s_{2}s_{1}\cdots s_{1}s_{2}$.
This is a $c$-sortable element.
We have $(\Pi/I_{c^{i}})e_{1}=(\Pi/J^{2i-1})e_{1}$, and $(\Pi/I_{c^{i}})e_{2}=(\Pi/J^{2i})e_{2}$, where $J$ is the Jacobson radical of $\Pi$.
By Theorem \ref{thm}, $M=\bigoplus_{i=1}^{n}(\Pi/I_{c^{i}})(i-1)$ is a tilting object in $\underline{\Sub}^{\mathbb{Z}}\Pi_w$, where graded projective $\Pi_{w}$-modules are removed.
The endomorphism algebra $\underline{\End}_{\Pi_w}^{\mathbb{Z}}(M)\simeq\End_{kQ}(M_{0})/[T]$ is given by the following quiver with relations
\[\Delta = 
\left[\xymatrix{1 \ar^{a}[r]<1.5pt> \ar_{b}[r]<-1.5pt> & 2 \ar^{a}[r]<1.5pt> \ar_{b}[r]<-1.5pt> & 3 \ar^(0.3){a}[r]<1.5pt> \ar_(0.3){b}[r]<-1.5pt> & \quad \cdots \quad \ar^(0.6){a}[r]<1.5pt> \ar_(0.6){b}[r]<-1.5pt> & 2n-1 \ar^(0.6){a}[r]<1.5pt> \ar_(0.6){b}[r]<-1.5pt> & 2n}\right], \quad aa=bb.
\]
The algebra $k\Delta/\langle aa-bb \rangle$ has global dimension two.
\end{example}
\subsection{Relationship between endomorphism algebras associated with $w$ and $w^{\prime}$}\label{subendalg}
In this subsection, we prove Theorem \ref{endred} which is used to prove Proposition \ref{propsur}.
Throughout this subsection, we use the notation in Definition \ref{def-p-m-M}.

Assume that $v$ is a source in $Q$.
Let $Q^{\prime}=\mu_{v} (Q)$ be the quiver obtained by reversing all arrows starting at $v$.
Although the preprojective algebras $\Pi$ and $\Pi^{\prime}$ of $Q$ and $Q^{\prime}$, respectively, are the same as ungraded algebras, they have different gradings.

We first construct a functor from $\mod^{\mathbb{Z}}\Pi$ to $\mod^{\mathbb{Z}}\Pi^{\prime}$.
Let $\beta_{1}, \beta_{2}, \ldots, \beta_{r}$ be the arrows in $Q$ starting at $v$, and
	\[
	Q_{1}^{\prime}=\big(Q_{1}\setminus \{\beta_{1}, \beta_{2}, \ldots, \beta_{r}\}\big) \sqcup \{ \gamma_{1}, \ldots, \gamma_{r} \},
	\]
where $t(\gamma_{i})=v$, $t(\beta_{i})=s(\gamma_{i})$.
We have an isomorphism of algebras $\rho : k\overline{Q} \to k\overline{Q^{\prime}}$ given by $\rho(\beta_{i})=\gamma_{i}^{\ast}, \rho(\beta_{i}^{\ast})=-\gamma_{i}$, and $\rho(\alpha)=\alpha$ for other arrows.
Then $\rho$ induces an isomorphism of the preprojective algebras, we also denote it by $\rho$:
\begin{align}\label{pipi}
	\rho : \Pi \xto{\sim} \Pi^{\prime}.
\end{align}
By calculating the grading of paths of $k\overline{Q}$ and $k\overline{Q^{\prime}}$,
we have the following lemma, where $\delta_{u,v}=1$ if $u=v$ and $0$ otherwise for $u,v \in Q_{0}$.
\begin{lemma}\label{gradeofpath}
For $u,u^{\prime} \in Q_{0}$ and $i\in\mathbb{Z}$, by identifying $kQ$ with $kQ^{\prime}$ by $\rho$,
we have
	\begin{align*}
	e_{u}(k\overline{Q})_{i}e_{u^{\prime}} =e_{u}(k\overline{Q^{\prime}})_{i+\delta_{u,v}-\delta_{u^{\prime},v}}e_{u^{\prime}}.
	\end{align*}
Moreover, the equation also holds for $\Pi$ and $\Pi^{\prime}$, that is,
	\begin{align*}
	e_{u}\Pi_{i}e_{u^{\prime}} = e_{u}\Pi^{\prime}_{i+\delta_{u,v}-\delta_{u^{\prime},v}}e_{u^{\prime}}.
	\end{align*}
\end{lemma}
For a finitely generated graded $\Pi^{\prime}$-module $N$, we regard $\End_{\Pi}(N)$ as a graded algebra by $\End_{\Pi}(N)_{i}=\Hom_{\Pi}^{\mathbb{Z}}(N,N(i))$.
The graded preprojective algebras $\Pi$ and $\Pi^{\prime}$ are related as follows.
\begin{lemma}
We have an isomorphism of graded algebras
	\begin{align*}
	\Pi^{\prime} \to \End_{\Pi}(\Pi e_{v}(1)\oplus\Pi(1-e_{v})), \hspace{0.5cm}x\mapsto(\cdot \rho^{-1}(x)).
	\end{align*}
\end{lemma}
\begin{proof}
It is enough to show that the morphism preserves gradings.
This follows from Lemma \ref{gradeofpath}.
\end{proof}
Then we construct a functor $\mathbb{G}$ from $\mod^{\mathbb{Z}}\Pi$ to $\mod^{\mathbb{Z}}\Pi^{\prime}$.
We need the following Lemma.
\begin{lemma}\label{idealend}
We have a surjective morphism of algebras $\Pi \to \End_{\Pi}(I_{v}), x\mapsto(\cdot x)$.
\end{lemma}
\begin{proof}
If $Q$ is a non-Dynkin quiver, then the assertion follows from Proposition \ref{birs} (a).
If $Q$ is a Dynkin quiver, then the assertion follows from \cite[Lemma 2.7]{M}.
\end{proof}
More precisely, we have the following surjective morphism of graded algebras.
\begin{lemma}\label{gradofepi}
Let $v \in Q_{0}$ be a source and $U:=I_{v}e_{v}(1)\oplus \Pi(1-e_{v}) \in \mod^{\mathbb{Z}}\Pi$.
Then we have a surjective morphism of graded algebras 
	\begin{align*}
	\Pi^{\prime} \to \End_{\Pi}(U), \hspace{0.5cm}x\mapsto(\cdot \rho^{-1}(x)).
	\end{align*}
Moreover, we have the following surjective morphism of graded algebras
	\begin{align*}
	\Pi^{\prime} \to \End_{\Pi}(\Pi/I_{v}), \hspace{0.5cm}x\mapsto(\cdot \pi(\rho^{-1}(x))),
	\end{align*}
where $\pi : \Pi \to \Pi/I_{v}$ is the canonical surjection.
\end{lemma}
\begin{proof}
The morphism is surjective since $\rho$ is an isomorphism and by Lemma \ref{idealend}.
We have to show that the composite is a morphism of graded algebras.

By Lemma \ref{gradeofpath}, for $u,u^{\prime}\in Q_{0}$, we have 
	\[
	 e_{u}(\Pi^{\prime}_{i})e_{u^{\prime}} = e_{u}(\Pi_{i+\delta_{u^{\prime},v}-\delta_{u, v}})e_{u^{\prime}}.
	\]
Moreover, for $u,u^{\prime}\in Q_{0}$ and $j\in\mathbb{Z}$, we have
\begin{align*}
U_{j}e_{u}\cdot e_{u}(\Pi_{i+\delta_{u^{\prime},v}-\delta_{u, v}})e_{u^{\prime}} &= \begin{cases}
				(I_{v})_{j+1}e_{u}\cdot e_{u}\Pi_{i}e_{u^{\prime}} & u=u^{\prime}=v\\
				(I_{v})_{j+1}e_{u}\cdot e_{u}\Pi_{i-1}e_{u^{\prime}} & u=v, u^{\prime}\neq v\\
				\Pi_{j}e_{u}\cdot e_{u}\Pi_{i+1}e_{u^{\prime}} & u\neq v, u^{\prime}=v\\
				\Pi_{j}e_{u}\cdot e_{u}\Pi_{i}e_{u^{\prime}} & u\neq v, u^{\prime}\neq v\\
					\end{cases}\\
				&\subset U(i)_{j}e_{u^{\prime}}.
\end{align*}
Thus, the morphism $\Pi^{\prime} \to \End_{\Pi}(U)$ is a morphism of graded algebras.
The other follows from a similar calculation.
\end{proof}
By Lemma \ref{gradofepi}, we have a functor 
		\[
		 \mathbb{G}:=\Hom_{\Pi}(U,-) : \mod^{\mathbb{Z}}\Pi \to \mod^{\mathbb{Z}}\Pi^{\prime}.
		\]
		where the grading on the $\Pi^{\prime}$-module $\mathbb{G}(X)$ is given by $\mathbb{G}(X)_{i}:=\Hom_{\Pi}^{\mathbb{Z}}(U,X(i))$.
This functor satisfies $\mathbb{G}\circ (i)\simeq (i)\circ \mathbb{G}$ for any $i\in\mathbb{Z}$.

To show Proposition \ref{Gto}, we recall the following proposition.
For a reduced expression $w=s_{u_1}s_{u_2}\cdots s_{u_l}$, let $I_{k,m}=I_{u_{k}\cdots u_{m}}$ if $k \leq m$ and $I_{k,m}=\Pi$ if $m < k$.
\begin{proposition}\cite[Lemma I\hspace{-.1em}I\hspace{-.1em}I. 1.14]{BIRSc}\label{birslem1.14}
Assume that $s_{u_1}s_{u_2}\cdots s_{u_l}$ is a reduced expression and $l\geq 2$.
Then we have $I_{k+1,m}/I_{1,m} \simeq \Hom_{\Pi}(\Pi/I_{u_{1}\cdots u_k}, \Pi/I_{u_{1}\cdots u_m})$ by $x \mapsto (\cdot x)$.
\end{proposition}
We apply the same construction as (\ref{m}) to the reduced expression $w^{\prime}:=s_{u_2}s_{u_3}\cdots s_{u_l}$.
Put
\begin{center}
$p_{u}^{\prime}=\Max\{2 \leq j \leq l \mid u_{j}=u\}-1,$ \hspace{0.4cm} for $u \in \supp(w^{\prime})$,\\
$m_{i}^{\prime} = \sharp\{ 2\leq j \leq i-1\mid u_j=u_i\},$ \hspace{0.4cm} for $2\leq i \leq l.$
\end{center}
Moreover, for $2 \leq i \leq l$, put
\[M^{\prime i-1}:=\left((\Pi^{\prime}/I^{\prime}_{u_2\cdots u_i})e_{u_{i}}\right)(m^{\prime}_{i}), \hspace{0.5cm} P^{\prime}=\bigoplus_{u\in\supp(w^{\prime})}M^{\prime p_{u}^{\prime}}.\]
We have $\Pi^{\prime}_{w^{\prime}}=\bigoplus_{u\in\supp(w^{\prime})}M^{\prime p^{\prime}_{u}}(-m^{\prime}_{p^{\prime}_{u}})$.
Put $M^{\prime}=\bigoplus_{i=2}^{l}M^{\prime i-1}$.
\begin{proposition}\label{Gto}
Assume that $w=s_{u_1}s_{u_2}\cdots s_{u_l}$ is a reduced expression of an element of $W_{Q}$ and $u_{1}=v$ is a source of $Q$, $l \geq 2$.
Let $w^{\prime}=s_{u_2}\cdots s_{u_l}$.
Then
\begin{itemize}
	\item[(a)]
	$\mathbb{G}(M^{1})=0$.
	
	\item[(b)]
	For $2\leq j \leq l$, we have an isomorphism $\psi_{j} : \mathbb{G}(M^{j})\xto{\sim} M^{\prime j-1}$ in $\mod^{\mathbb{Z}}\Pi^{\prime}$, that is,
	\[
    	\psi_{j} : \mathbb{G}\left((\Pi/I_{1, j})e_{u_{j}}\right)(m_{j}) \xto{\sim} \left((\Pi^{\prime}/I^{\prime}_{2, j})e_{u_{j}}\right)(m_{j}^{\prime}).
	\]
	\if()
	Let $m$ be the number of the set $\{ 1\leq k \leq l-1\mid u_k=u_l\}$ and 
	let $m^{\prime}$ be the number of the set $\{ 2\leq k \leq l-1\mid u_k=u_l\}$.
	Then we have an isomorphism in $\mod^{\mathbb{Z}}\Pi^{\prime}$
	\[
	\mathbb{G}\left((\Pi/I_{w})e_{u_{l}}(m)\right) \simeq (\Pi^{\prime}/I^{\prime}_{w^{\prime}})e_{u_{l}}(m^{\prime}).
	\]\fi
	
	\item[(c)]We have 
	$\psi=\bigoplus_{j=1}^{l}\psi_{j} : \mathbb{G}(M)=\mathbb{G}(M/M^{1}) \xto{\sim} M^{\prime}$ in $\mod^{\mathbb{Z}}\Pi^{\prime}$.
\end{itemize}
\end{proposition}
\begin{proof}
(\rm a)
Since a simple module associated with $u_{1}=v$ does not appear in $\top(U)$, we have $\Hom_{\Pi}(U,M^{1})=0$.

(\rm b)
Since $m_{j}^{\prime}=m_{j}-\delta_{v, u_{j}}$ holds,
we show that 
\[
\mathbb{G}\left((\Pi/I_{1, j})e_{u_{j}}\right) \simeq \left((\Pi^{\prime}/I^{\prime}_{2, j})e_{u_{j}}\right)(-\delta_{v, u_{j}}).
\]
By a similar calculation of the proof of Lemma \ref{gradofepi}, we have the following morphism of graded $\Pi^{\prime}$-modules
\begin{align*}
 &(I_{2, j}^{\prime}/I_{1, j}^{\prime})(-\delta_{v, u_{j}}) \to \Hom_{\Pi}((\Pi/I_{v})(1), \Pi/I_{1, j}),\\
& (\Pi^{\prime}/I_{1, j}^{\prime})(-\delta_{v, u_{j}}) \to \Hom_{\Pi}(\Pi e_{v}(1)\oplus\Pi(1-e_{v}), \Pi/I_{1, j}),
\end{align*}
where both of them are defined by $x \mapsto (\cdot \rho^{-1}(x))$.
These morphisms are isomorphisms by Proposition \ref{birslem1.14}.
By Proposition \ref{birs} (e), $\Ext_{\Pi}^{1}(\Pi/I_{v}, \Pi/I_{1,j})=0$ holds.
Applying the functor $\Hom_{\Pi}(-,\Pi/I_{1, j})$ to the exact sequence \[0\to U\to\Pi e_{v}(1)\oplus\Pi(1-e_{v})\to(\Pi/I_{v})(1)\to0,\] we have the following commutative diagram of exact sequence in $\mod^{\mathbb{Z}}\Pi^{\prime}$;

\xymatrix{
0 \ar[r] & (I_{2,j}^{\prime}/I_{1,j}^{\prime})(-\delta_{v, u_{j}}) \ar[r] \ar[d]^{\simeq} & (\Pi^{\prime}/I_{1, j}^{\prime})(-\delta_{v, u_{j}}) \ar[r] \ar[d]^{\simeq} & (\Pi^{\prime}/I_{2, j}^{\prime})(-\delta_{v, u_{j}}) \ar[r] \ar[d] & 0 \\
0 \ar[r] & {}_{\Pi}((\Pi/I_{v})(1), \Pi/I_{1, j}) \ar[r] & {}_{\Pi}(\Pi e_{v}(1)\oplus\Pi(1-e_{v}), \Pi/I_{1, j}) \ar[r] & {}_{\Pi}(U, \Pi/I_{1, j}) \ar[r] & 0.}
Therefore we have the assertion.

(\rm c)
This comes from (a) and (b).
\end{proof}

The following lemma is used later.
\begin{lemma}\label{Gcor}
Under the setting in Proposition \ref{Gto},
for the functor $\mathbb{G} : \mod^{\mathbb{Z}}\Pi \to \mod^{\mathbb{Z}}\Pi^{\prime}$, we have
\begin{itemize}
\item[(a)]
$\mathbb{G}$ restricts to a dense functor $\proj^{\mathbb{Z}}\Pi_{w}$ to $\proj^{\mathbb{Z}}\Pi^{\prime}_{w^{\prime}}$.
\item[(b)]
For $i\in \mathbb{Z}$, the map $\mathbb{G}_{M,M(i)}$ is surjective.
\end{itemize}
\end{lemma}
\begin{proof}
(\rm a)
This comes from $\Pi_{w}=\bigoplus_{u \in \supp(w)}M^{p_{u}}(-m_{p_{u}})$, $\Pi^{\prime}_{w^{\prime}}=\bigoplus_{u\in\supp(w^{\prime})}M^{\prime p^{\prime}_{u}}(-m^{\prime}_{p^{\prime}_{u}})$, and Proposition \ref{Gto}.

(\rm b)
It is enough to show that the map $\mathbb{G}_{M^{j},M^{k}(i)}$ is surjective for $2\leq j,k\leq l$.
By Lemma \ref{birslem1.14} (b), 
we have 
	\begin{align*}
	\Hom_{\Pi}^{\mathbb{Z}}(M^{j},M^{k}(i))=\left( e_{u_j}\frac{I_{j+1,k}}{I_{1,k}}e_{u_{k}} \right)_{m_k-m_j+i},\\
	\Hom_{\Pi^{\prime}}^{\mathbb{Z}}(M^{\prime j-1},M^{\prime k-1}(i))=\left( e_{u_{j}}\frac{I^{\prime}_{j+1,k}}{I^{\prime}_{2,k}}e_{u_{k}} \right)_{m_k^{\prime}-m^{\prime}_j+i}.
	\end{align*}
For $2\leq j,k\leq l$, an equation $m_{k}^{\prime}-m_{j}^{\prime}+i=m_{k}-m_{j}+\delta_{u_{j},v}-\delta_{u_{k},v}+i$ holds.
Thus $\rho : \Pi\to\Pi^{\prime}$ maps $\left( e_{u_{j}}(I_{j+1,k}/I_{1,k})e_{u_{k}} \right)_{m_k-m_j+i}$ to $\left( e_{u_{j}}(I^{\prime}_{j+1,k}/I^{\prime}_{2,k})e_{u_{k}} \right)_{m_k^{\prime}-m^{\prime}_j+i}$ by Lemma \ref{gradeofpath}.
We have the following commutative diagram
	\begin{align}\label{diagramHI}
	\xymatrix{
	\Hom_{\Pi}^{\mathbb{Z}}(M^{j},M^{k}(i)) \ar[r]^/-0.4cm/{\mathbb{G}_{M^{j},M^{k}(i)}} \ar[d]^{\simeq} & \Hom_{\Pi^{\prime}}^{\mathbb{Z}}(\mathbb{G}(M^{j}),\mathbb{G}(M^{k})(i)) \ar[r]_{\sim}^{\alpha} &\Hom_{\Pi^{\prime}}^{\mathbb{Z}}(M^{\prime j-1},M^{\prime k-1}(i)) \ar[d]^{\simeq} \\
	\left( e_{u_j}\frac{I_{j+1,k}}{I_{1,k}}e_{u_{k}} \right)_{m_k-m_j+i} \ar[rr] & & \left( e_{u_{j}}\frac{I^{\prime}_{j+1,k}}{I^{\prime}_{2,k}}e_{u_{k}} \right)_{m_k^{\prime}-m^{\prime}_j+i},
	}
	\end{align}
where the lower map is induced by $\rho : \Pi \to \Pi^{\prime}$, and $\alpha$ is defined by $\alpha(f)=\psi_{k}(i)\circ f \circ \psi_{j}^{-1}$.
Since the lower map is surjective and $\alpha$ is an isomorphism by Proposition \ref{Gto} (b), we have that $\mathbb{G}_{M^{j},M^{k}(i)}$ is surjective.
\end{proof}
The following theorem is a graded version of \cite[Theorem 3.1, (ii)]{IR} and the main theorem of this subsection.
\begin{theorem}\label{endred}
Under the setting in Proposition \ref{Gto},
we have an isomorphism of algebras
\[
\underline{G} : \End_{\Pi_{w}}^{\mathbb{Z}}(M)/[M^{1}(i) \mid 0 \leq i \leq p_{u_{1}}] \xto{\sim} \End_{\Pi_{w^{\prime}}^{\prime}}^{\mathbb{Z}}(M^{\prime}),
\]
where $G(-)=\psi\circ{\mathbb{G}}_{M,M}(-)\circ\psi^{-1}$ and $[M^{1}(i) \mid 0 \leq i \leq p_{u_{1}}]$ is an ideal of $\End_{\Pi_{w}}^{\mathbb{Z}}(M)$ consisting of morphisms factoring through objects in $\add \{M^{1}(i) \mid 0\leq i \leq p_{u_{1}} \}$.
\end{theorem}
\begin{proof}
We show that $G$ is surjective and $\Ker(G)=[M^{1}(i) \mid 0 \leq i \leq p_{u_{1}}]$.

(\rm i)
By Lemma \ref{Gcor} (b), $G$ is surjective.

(\rm ii)
Since $\psi$ is an isomorphism, we have $\Ker(G)=\Ker(\mathbb{G}_{M,M})$.
We show that $\Ker (\mathbb{G}_{M,M})=[M^{1}(i) \mid 0 \leq i \leq p_{u_{1}}]$.
By Proposition \ref{Gto} (a), we have $[M^{1}(i) \mid 0 \leq i \leq p_{u_{1}}]\subset\Ker \left( \mathbb{G}_{M^{j},M^{k}} \right)$.
Conversely, we show that $\Ker \left( \mathbb{G}_{M^{j},M^{k}} \right)\subset[M^{1}(i) \mid 0 \leq i \leq p_{u_{1}}]$ for $2\leq j, k\leq l$.
By the commutative diagram (\ref{diagramHI}), we have 
\begin{align*}
\Ker \left( \mathbb{G}_{M^{j},M^{k}} \right) =\left( e_{u_j}\frac{I_{2,k}}{I_{1,k}}e_{u_{k}} \right)_{m_k-m_j}.
\end{align*}
If $u_{j}\neq u_{1}$, then $e_{u_{j}}I_{1, k}=e_{u_{j}}I_{2, k}$ and we have $\Ker \left( \mathbb{G}_{M^{j},M^{k}} \right)=0$.
If $u_{j}=u_{1}$, then we have 
\begin{align*}
\Ker \left( \mathbb{G}_{M^{j},M^{k}} \right) & = \left( e_{u_j}\frac{I_{2,k}}{I_{1,k}}e_{u_{k}} \right)_{m_k-m_j} \\
						& = \left( e_{u_{j}}\frac{\Pi}{I_{u_1}}e_{u_{1}} \right) \left( e_{u_1}\frac{I_{2,k}}{I_{1,k}}e_{u_{k}} \right)_{m_k-m_j}\\
						& =\Hom_{\Pi}^{\mathbb{Z}}(M^{1}(m_{j}),M^{k}) \circ \Hom_{\Pi}^{\mathbb{Z}}(M^{j},M^{1}(m_{j})).
\end{align*}
In particular, we have $\Ker \left( \mathbb{G}_{M^{j},M^{k}} \right)\subset[M^{1}(i) \mid 0 \leq i \leq p_{u_{1}}]$.
\end{proof}
We end this subsection by showing the following lemma which is used later to show Lemma \ref{diagram}.
For a source $v\in Q_{0}$ and $Q^{\prime}=\mu_{v}(Q)$, we have the reflection functor
\[ \mod\,kQ \xto{R^{+}_{v}} \mod\,kQ^{\prime}. \]
Note that $U$ is generated by $U_{0}$ as a left $\Pi$-module.
In fact, $I_{v}e_{u}=\Pi e_{u}$ is generated by $e_{u}$ for $u\not=v$ and $I_{v}e_{v}$ is generated by all arrows in $\overline{Q}$ starting at $v$.We denote by $\mathbb{F}^{\prime}$ the degree zero functor on $\mod^{\mathbb{Z}}\Pi^{\prime}$:
\[
\mathbb{F}^{\prime}=(-)_{0}:\mod^{\mathbb{Z}}\Pi^{\prime}\to\mod kQ^{\prime}.
\]
\begin{lemma}\label{G0R0}
Let $v$ be a source of $Q$  and $Q^{\prime}=\mu_{v}(Q)$.
\begin{itemize}
\item[(a)]
We have a morphism of functors $\phi :  \mathbb{F}^{\prime}\circ\mathbb{G}\to R_{u_{1}}^{+}\circ\mathbb{F}$.
\item[(b)]
For any $X\in\mod^{\leq 0}\Pi$, $\phi_{X}:\mathbb{G}(X)_{0}\to R_{v}^{+}(X_{0})$ is an isomorphism of $kQ^{\prime}$-modules, that is, the following diagram  of functors is commutative on $\mod^{\leq 0}\Pi$:
\[
\xymatrix{
\mod^{\mathbb{Z}}\Pi \ar[r]^{\mathbb{G}} \ar[d]^{\mathbb{F}} & \mod^{\mathbb{Z}}\Pi^{\prime} \ar[d]^{\mathbb{F}^{\prime}} \\
\mod\,kQ \ar[r]^{R^{+}_{v}} & \mod\,kQ^{\prime}.
}
\]
\end{itemize}
\end{lemma}
\begin{proof}
By the definition of the functor $\mathbb{G}$, we have $\mathbb{G}(X)_{0}=\Hom_{\Pi}^{\mathbb{Z}}(U,X)$.
Since $\Pi_{i}\simeq\tau^{-i}(kQ)$ as $kQ$-modules, $U_{0}=\tau^{-}(kQe_{v})\oplus kQ(1-e_{v})$ holds and this is an APR-tilting $kQ$-module associated with $v$.
Therefore we have a morphism of $kQ^{\prime}$-modules
\[\phi_{X} : \mathbb{G}(X)_{0}=\Hom_{\Pi}^{\mathbb{Z}}(U,X) \to \Hom_{kQ}(U_{0},X_{0})=R_{v}^{+}(X_{0}),\]
given by $\phi_{X}(f)=f|_{U_{0}}$.
Clearly this gives a morphism $\phi :  \mathbb{F}^{\prime}\circ\mathbb{G}\to R_{u_{1}}^{+}\circ\mathbb{F}$ of functors.
Since $U$ is generated by $U_{0}$ as a graded $\Pi$-module, a morphism $f\in\Hom_{\Pi}^{\mathbb{Z}}(U,X)$ is determined by $\phi_{X}(f)$.
This implies that $\phi_{X}$ is injective.

We show that $\phi_{X}$ is surjective when $X$ is in $\mod^{\leq 0}\Pi$.
Let $g\in\Hom_{kQ}(U_{0},X_{0})$.
We define a morphism $f : U\to X$ of  $kQ$-modules by $f|_{U_{0}}=g$ and $f|_{U_{\geq 1}}=0$.
Then $f$ gives a morphism in $\mod^{\mathbb{Z}}\Pi$, since $X\in\mod^{\leq 0}\Pi$ and $\Pi$ is positively graded.
\end{proof}
\subsection{\underline{F}\,is surjective}\label{uF-surj}
We use the notation in Subsection \ref{welldef} and \ref{subendalg}.
For a quiver $Q$, we denote by $W_{Q}$ the Coxeter group of $Q$.
Assume that $w=c^{(0)}c^{(1)}\cdots c^{(m)}=s_{u_1}s_{u_2}\cdots s_{u_l}$ is a $c$-sortable element of $W_{Q}$.
Without loss of generality by Lemma \ref{airtlem2.1}, we assume that $Q_{0}=\supp(w)$.
Let $Q^{\prime}=\mu_{u_1}(Q)$.
We show that the morphism $\underline{F} : \underline{\End}_{\Pi_w}^{\mathbb{Z}}(M) \to \End_{kQ}(M_0)/[T]$ is surjective.
We first prove the following lemma.
\begin{lemma}\label{sur}
An element $w^{\prime}=s_{u_2}\cdots s_{u_l}$ is a $(s_{u_1}cs_{u_1})$-sortable element in $W_{Q^{\prime}}$.
\end{lemma}
\begin{proof}
It is clear that $s_{u_1}cs_{u_1}$ is a Coxeter element of $W_{Q^{\prime}}$ admissible with respect to the orientation of $Q^{\prime}$.
Let $a=\Max\{k \mid u_1 \in \supp(c^{(k)})\}$.
Put
	\begin{align}
	c^{\prime(k)} = \begin{cases}
				s_{u_1}c^{(k)}s_{u_1} & 0 \leq k \leq a-1\\
				s_{u_1}c^{(k)} & k=a\\
				c^{(k)} & a+1 \leq k \leq m.
			\end{cases}\notag
	\end{align}
Then we have a reduced expression  $w^{\prime}=c^{\prime(0)}c^{\prime(1)}\cdots c^{\prime(m^{\prime})}$, where $m^{\prime}=m-1$ if $\supp(c^{(m)})=\{ u_{1}\}$, and $m^{\prime}=m$ if otherwise.
Since each $c^{\prime(k)}$ is a subword of $s_{u_1}cs_{u_1}$, 
$w^{\prime}$ is a $(s_{u_1}cs_{u_1})$-sortable element.
\end{proof}
Let $w^{\prime}=s_{u_2}\cdots s_{u_l}$.
By Proposition \ref{Gto} (c), there exists the isomorphism of graded $\Pi^{\prime}$-modules \[\psi : \mathbb{G}(M/M^{1}) \xto{\sim} M^{\prime}.\]
By using $\psi$, we have an isomorphism of algebras 
\[\alpha : \End_{\Pi^{\prime}}^{\mathbb{Z}}(\mathbb{G}(M/M^{1})) \to \End_{\Pi^{\prime}}^{\mathbb{Z}}(M^{\prime })\]
defined by $\alpha(f)=\psi \circ f \circ \psi^{-1}$.
Moreover we have an isomorphism of algebras
\[\alpha_{0} : \End_{kQ^{\prime}}(\mathbb{G}(M/M^{1})_{0}) \to \End_{kQ^{\prime}}(M^{\prime }_{0})\]
defined by $\alpha_{0}(f)=\psi_{0} \circ f \circ \psi_{0}^{-1}$, where $\psi_{0}=\psi|_{\mathbb{G}(M/M^{1})_{0}}$.
Let 
 \[F_{>1}:=\mathbb{F}_{M/M^{1},M/M^{1}}:\End_{\Pi}^{\mathbb{Z}}(M/M^{1}) \to \End_{kQ}\left( (M/M^{1})_{0} \right).\]
\begin{lemma}\label{diagram}
The following diagram is commutative: 
\begin{align}\label{comdiagram1}
	\xymatrix{\End_{\Pi}^{\mathbb{Z}}(M/M^{1}) \ar[r]^{G_{>1}} \ar[d]^{F_{>1}} & \End_{\Pi^{\prime}}^{\mathbb{Z}}(\mathbb{G}(M/M^{1})) \ar[d]^{\overline{F}^{\prime}} \ar[r]^/0.2cm/{\alpha}_/0.2cm/{\sim} & \End_{\Pi^{\prime}}^{\mathbb{Z}}(M^{\prime }) \ar[d]^{F^{\prime}} \\ 
	\End_{kQ}((M/M^{1})_{0}) \ar[r]^{R}_{\sim} & \End_{kQ^{^{\prime}}}(\mathbb{G}(M/M^{1})_{0}) \ar[r]^/0.2cm/{\alpha_{0}}_/0.2cm/{\sim} & \End_{kQ^{^{\prime}}}(M_{0}^{\prime}),}
	\end{align}
where $G_{>1}=\mathbb{G}_{M/M^{1},M/M^{1}}$, $\overline{F}^{\prime}=\mathbb{F}^{\prime}_{\mathbb{G}(M/M^{1}),\mathbb{G}(M/M^{1})}$, and
$R$ is defined by $R(f)=(\phi_{{M/M^{1}}})^{-1} \circ R_{u_{1}}^{+}(f)\circ \phi_{{M/M^{1}}}$. 
\end{lemma}
\begin{proof}
The commutativity of the left square comes from the functoriality of $\phi$ of Lemma \ref{G0R0}.
The commutativity of the right square is clear.
\end{proof}
\begin{proposition}\label{propsur}
Assume that $w=s_{u_1}s_{u_2}\cdots s_{u_l}=c^{(0)}c^{(1)}\cdots c^{(m)}$ is a $c$-sortable element of $W_{Q}$.
Then we have
\begin{itemize}
\item[(a)]
The morphism $F : \End_{\Pi_w}^{\mathbb{Z}}(M) \to \End_{kQ}(M_0),\,f \mapsto f|_{M_0}$ is surjective. 

\item[(b)]
The morphism $\underline{F} : \underline{\End}_{\Pi_w}^{\mathbb{Z}}(M) \to \End_{kQ}(M_0)/[T]$ is surjective.
\end{itemize}
\end{proposition}
\begin{proof}
(\rm a) We show the assertion by induction on $l$.
Assume that $l=1$.
Then we have $M=M^{1}=M^{1}_{0}$ and $\Pi_{w}=kQ$.
Thus we have $\End_{\Pi_w}^{\mathbb{Z}}(M) = \End_{kQ}(M_0)$.
The assertion holds.
Assume that $l\geq 2$.
We show that two maps
\begin{align*}
F_{1}:=\mathbb{F}_{M^{1},M}: \Hom_{\Pi_{w}}^{\mathbb{Z}}(M^{1},M) \to \Hom_{kQ}(M^{1}_{0},M_{0}),\\
\mathbb{F}_{M/M^{1},M}:\Hom_{\Pi_{w}}^{\mathbb{Z}}(M/M^{1},M) \to \Hom_{kQ}\left( (M/M^{1})_{0},M_{0} \right)
\end{align*}
are surjective.
Since $M^{1}=M^{1}_{0}$, $M$ is in $\mod^{\leq 0}\Pi_{w}$, and $\Pi_{w}$ is positively graded, we can regard any $g\in\Hom_{kQ}(M^{1}_{0},M_{0})$ as a morphism in $\mod^{\mathbb{Z}}\Pi_{w}$.
Therefore, $F_{1}$ is surjective.

By \cite[Corollary 3.10]{AIRT}, we have $\Hom_{kQ}(M^{j}_{0}, M^{i}_{0})=0$ for $i<j$.
Thus we have $\Hom_{kQ}\left( (M/M^{1})_{0},M_{0} \right)=\End_{kQ}\left( (M/M^{1})_{0} \right)$.
Therefore it is enough to show that the map \[F_{>1}:=\mathbb{F}_{M/M^{1},M/M^{1}}:\End_{\Pi_{w}}^{\mathbb{Z}}(M/M^{1}) \to \End_{kQ}\left( (M/M^{1})_{0} \right)\] is surjective.
We show that $F_{> 1}$ is surjective by using the diagram (\ref{comdiagram1}).
Let $w^{\prime}=s_{u_2}\cdots s_{u_l}$.
By Lemma \ref{sur} (c), $w^{\prime}$ is a $(s_{u_1}cs_{u_1})$-sortable element in $W_{Q^{\prime}}$.
Thus, by the inductive hypothesis, $F^{\prime}$ in the diagram (\ref{comdiagram1}) is surjective.
By Theorem \ref{endred}, $G_{>1}$ is surjective.
Since $\alpha$, $\alpha_{0}$, and $R$ are isomorphism, $F_{>1}$ is surjective.

(\rm b) We have the following commutative diagram 
	\begin{align}\label{FF}
	\xymatrix{
	\End_{\Pi_{w}}^{\mathbb{Z}}(M) \ar[r]^{\pi} \ar[d]^{F} & \underline{\End}_{\Pi_{w}}^{\mathbb{Z}}(M) \ar[d]^{\underline{F}}\\ 
	 \End_{kQ}(M_{0}) \ar[r]^/-0.2cm/{\pi^{\prime}} & \End_{kQ}(M_{0})/[T].
	}
	\end{align}
Since the bottom and the left morphisms are surjective, the right morphism is surjective.
\end{proof}
\subsection{\underline{F}\,is injective}\label{uF-inj}
We show that the morphism $\underline{F}$ is injective.
Let $w=s_{u_{1}}s_{u_{2}}\cdots s_{u_{l}}$ be a $c$-sortable element and $w^{\prime}=s_{u_{2}}\cdots s_{u_{l}}$.
Without loss of generality by Lemma \ref{airtlem2.1}, we assume that $Q_{0}=\supp(w)$.
Since $\mathbb{G}(M^{1})=0$ and by Lemma \ref{diagram}, we have the following commutative diagram:
	\begin{align}\label{comdiagram2}
	\xymatrix{\End_{\Pi_{w}}^{\mathbb{Z}}(M) \ar[r]^{G} \ar[d]^{F} & \End_{\Pi_{w^{\prime}}^{\prime}}^{\mathbb{Z}}(M^{\prime}) \ar[d]^{F^{\prime}} \\ 
	\End_{kQ}(M_{0}) \ar[r]^{\overline{R}} & \End_{kQ^{\prime}}(M_{0}^{\prime}),}
	\end{align}
where $\overline{R}=\alpha_{0}\circ R$.
\begin{lemma}\label{f1f2f3}
Let $f\in\End_{\Pi_{w}}^{\mathbb{Z}}(M)$.
Assume that $G(f)$ factors through $\add P^{\prime}$.
Then we have
\begin{itemize}
\item[(a)] $f$ factors through $\add(P\oplus M^{1})$.
\item[(b)] If $F(f)=0$, then $f$ factors through $\add(P)$.
\end{itemize}
\end{lemma}
\begin{proof}
(\rm a)
By Proposition \ref{Gto} (d), we have $\mathbb{G}(P)=P^{\prime}$.
Since $G(f)$ factors through $\add P^{\prime}$ and by Theorem \ref{endred} and Lemma \ref{Gcor}, there exist morphisms $f_{1}, g\in\End_{\Pi_{w}}^{\mathbb{Z}}(M)$ such that $f=f_{1}+g$, $f_{1}$ factors through $\add P$, and $g$ factors through $\add\{M^{1}(i)\mid i\geq 0\}$.
Thus $g$ is the sum of morphisms $g_{1}, g_{2}\in\End_{\Pi_{w}}^{\mathbb{Z}}(M)$ such that $g_{1}$ factors through $\add M^{1}$ and $g_{2}$ factors through $\add\{M^{1}(i)\mid i\geq 1\}$.
By Lemma \ref{facthr}, $g_{2}$ factors through $\add P$.

(\rm b)
By (a), there exists $g\in\End_{\Pi_{w}}^{\mathbb{Z}}(M)$ such that $g$ factors through $\add M^{1}$ and $f-g$ factors through $\add P$.
We show that $g$ factors through $\add P$.
Since $\Hom_{\Pi_{w}}^{\mathbb{Z}}(M/M^{1},M^{1})=0$, we have $g|_{M/M^{1}}=0$.
Therefore we may regard $g$ as a morphism from $M^{1}$ to $M$.
Since $F(f)=0$, $F(g-f)=F(g) : M^{1}_{0} \to M_{0}$ factors through $\add P_{0}$.
By Proposition \ref{propsur} (a), there exists $h\in\Hom_{\Pi_{w}}^{\mathbb{Z}}(M^{1},M)$ such that $h$ factors through $\add P$ and $F(g)=F(h)$.
Because $M^{1}=M^{1}_{0}$, we have $g=h$.
\end{proof}
\begin{proposition}\label{prop-uF-inj}
The morphism $\underline{F} : \underline{\End}_{\Pi_w}^{\mathbb{Z}}(M) \to \End_{kQ}(M_0)/[T]$ is injective.
\end{proposition}
\begin{proof}
We show the assertion by induction on $l$.
If $l=1$, then we have $\underline{\End}_{\Pi_w}^{\mathbb{Z}}(M) = \End_{kQ}(M_0)/[T]=0$.
Thus the claim is clear.

Assume that $l\geq 2$.
Let $f$ be a morphism in $\End_{\Pi_{w}}^{\mathbb{Z}}(M)$ satisfying $\underline{F}(\pi(f))=0$.
We show $\pi(f)=0$.
By the commutative diagram (\ref{FF}), we have $\pi^{\prime}(F(f))=0$.
Since $\Ker \pi^{\prime}=[T]$, $F(f)$ factors through $\add T$.
By Proposition \ref{propsur} (a) and $\mathbb{F}(P)=T$, there exists $g\in\End_{\Pi_{w}}^{\mathbb{Z}}(M)$ such that $g$ factors through $\add P$ and $F(f)=F(g)$.
Put $h:=f-g\in\End_{\Pi_{w}}^{\mathbb{Z}}(M)$.
We have $\pi(f)=\pi(h)$.
Therefore it is enough to show $\pi(h)=0$.

Consider the following commutative diagram
	\begin{align*}
	\xymatrix{\End_{\Pi_{w}}^{\mathbb{Z}}(M) \ar[r]^{G} \ar[d]^{F}&
	\End_{\Pi^{\prime}_{w^{\prime}}}^{\mathbb{Z}}(M^{\prime}) \ar[r]^{\eta} \ar[d]^{F^{\prime}} & \underline{\End}_{\Pi_{w^{\prime}}^{\prime}}^{\mathbb{Z}}(M^{\prime}) \ar[d]^{\underline{F^{\prime}}}\\ 
	  \End_{kQ}(M_{0}) \ar[r]^{\overline{R}} &  \End_{kQ^{\prime}}(M^{\prime}_{0}) \ar[r]^/-0.4cm/{\eta^{\prime}} & \End_{kQ^{\prime}}(M_{0}^{\prime})/[T^{\prime}],
	}
	\end{align*}
where $\eta$ and $\eta^{\prime}$ are canonical surjections.
We have \[\underline{F^{\prime}}(\eta(G(h)))=\eta^{\prime}(F^{\prime}(G(h)))=\eta^{\prime}(R(F(h)))=0,\] since $F(h)=F(f)-F(g)=0$.
By the inductive hypothesis, $\underline{F^{\prime}}$ is injective.
Thus $\eta(G(h))=0$ and $G(h)$ factors through a graded projective $\Pi^{\prime}_{w^{\prime}}$-module.
By the proof of Proposition \ref{induced}, $G(h)$ factors through $\add P^{\prime}$.
Thus, by Lemma \ref{f1f2f3} (b), $h$ factors through $\add P$.
Therefore, we have $\pi(f)=\pi(h)=0$.
\end{proof}
\section{The global dimension of the endomorphism algebra}\label{gldimofend}
Throughout this section, let $A$ be a finite dimensional algebra and $T$ a {\it cotilting} $A$-module of finite injective dimension, that is, $T$ satisfies $\injdim T<\infty$, $\Ext_{A}^{i}(T,T)=0$ for any $i>0$, and there exists an exact sequence $0\to T_{r}\to \cdots\to T_{1}\to T_{0}\to\kD A\to0$ where $T_{i}\in\add T$.
We denote by ${}^{\perp_{>0}}T$ the full subcategory consisting of $\mod A$ of modules $X$ satisfying $\Ext_{A}^{i}(X,T)=0$ for any $i>0$.
The aim of this section is to show the following theorem.
\begin{theorem}\label{end2}
Assume that the global dimension of $A$ is at most $n$ and that ${}^{\perp_{>0}}T$ has an additive generator $M$.
Then the global dimension of $\End_{A}(M)/[T]$ is at most $3n-1$.
\end{theorem}
Note that $\End_{A}(M)$ and $\End_{A}(M)/[T]$ are relative version of Auslander algebras and stable Auslander algebras.
It is known that Auslander algebras have global dimension at most two \cite{ARS}, and that stable Auslander algebras have global dimension at most $3(\gl A)-1$ \cite[Proposition 10.2]{AR74}.
We apply Theorem \ref{end2} to our endomorphism algebra in Theorem \ref{endalg}.
We denote by $\Sub T$ the full subcategory of $\mod A$ consisting of submodules of finite direct sums of $T$.
\begin{corollary}\label{end2cor}
Under the setting in Theorem \ref{endalg}, the global dimension of $\End_{kQ}(M_0)/[T]$ is at most two.
\end{corollary}
\begin{proof}
Let $Q^{(1)}$ be the full subquiver of $Q$ whose the set of vertices is $\supp(w)$.
We have $\End_{kQ}(M_0)/[T]=\End_{kQ^{(1)}}(M_0)/[T]$.
Moreover, by Theorem \ref{airt}, $T$ is a tilting $kQ^{(1)}$-module.
By \cite[Theorem 3.11]{AIRT}, we have $\Sub{T}=\add \{ M_{0}^{1}, M_{0}^{2}, \ldots, M_{0}^{l} \}$.
By Bongartz's lemma \cite[Chapter VI, 2.4. Lemma]{ASS}, tilting modules over a hereditary algebra coincide with cotilting modules.
Since $kQ^{(1)}$ is hereditary, $\Sub T={}^{\perp_{>0}}T$ holds.
Therefore, by applying Theorem \ref{end2}, the global dimension of $\End_{kQ^{(1)}}(M_0)/[T]$ is at most two.
\end{proof}
To show Theorem \ref{end2}, we use cotilting theory.
We recall some properties of cotilting modules.
\begin{proposition}\cite[Theorem 5.4, Proposition 5.11]{AR91}\label{cotilt}
Let $T$ be a cotilting $A$-module.
Then
\begin{itemize}
\item[(a)] For any $X \in  {}^{\perp_{>0}}T$, there exists an injective left $(\add T)$-approximation of $X$. 
\item[(b)] Let $X\in {}^{\perp_{>0}}T$. Then $X\in \add T$ if and only if $\Ext_{A}^{1}(Y,X)=0$ for any $Y\in{}^{\perp_{>0}}T$.
\end{itemize}
\end{proposition}
In the following lemma and proposition, we construct an important long exact sequence.
For $X,Y\in \mod A$, we denote by $\overline{\Hom}_{A}^{T}(X,Y)$ the quotient of $\Hom_{A}(X,Y)$ by the subspace consisting of morphisms factoring through $\add T$, that is, $\overline{\Hom}_{A}^{T}(X,Y)=\Hom_{A}(X,Y)/[T]$.
\begin{lemma}\label{exact}
For an exact sequence $0\to X \xto{f}Y \xto{g}Z\to0$ in ${}^{\perp_{>0}}T$ and any $A$-module $N$,
we have the following exact sequence
\[
\overline{\Hom}_{A}^{T}(Z,M)\xto{-\circ g} \overline{\Hom}_{A}^{T}(Y,M)\xto{-\circ f} \overline{\Hom}_{A}^{T}(X,M).
\]
\end{lemma}
\begin{proof}
It is enough to show that $\Ker(-\circ f) \subset \Im(-\circ g)$.
Assume that $\alpha\in\Hom_{A}(Y,N)$ satisfies $f\alpha=0\in\Hom_{A}(X,N)/[T]$.
There exists a module $T^{\prime}\in\add T$ and morphisms $h_{1}:X\to T^{\prime}$, $h_{2}:T^{\prime}\to N$ such that $f\alpha =h_{1}h_{2}$.
Since $\Ext_{A}^{1}(Z,T)=0$, there exists a morphism $\beta: Y \to T^{\prime}$ such that $f\beta=h_{1}$.
Since $f(\alpha - \beta h_{2})=f\alpha -f\beta h_{2}=f\alpha-h_{1}h_{2}=0$, there exists a morphism $\gamma:Z\to M$ such that $g\gamma=\alpha -\beta h_{2}$.
	\begin{align*}
	\xymatrix{
	X\ar[r]^{f}\ar[d]_{h_{1}} & Y \ar[r]^{g}\ar[d]^/-0.2cm/{\alpha}\ar@{-->}[dl]_{\beta} & Z \ar@{-->}[dl]^{\gamma} \\
	T^{\prime} \ar[r]_{h_{2}} & N.
	}
	\end{align*}
\end{proof}
Let $X\in{}^{\perp_{>0}}T$.
By Proposition \ref{cotilt} (b), there exists an injective left $(\add T)$-approximation $f:X\to T^{\prime}$.
We have $\Cok f\in{}^{\perp_{>0}}T$.
We denote by $\Omega_{T}^{-}(X)$ a cokernel of $f$.
Note that $\Omega_{T}^{-}(X)$ is uniquely determined by $X$ up to direct summands in $\add T$.
Let $\Omega_{T}^{-n}(X)=\Omega_{T}^{-}(\Omega_{T}^{-(n-1)}(X))$ for $n>1$.
\begin{proposition}\label{longexa}
Let $0\to X \xto{f}Y \xto{g}Z\to0$ be an exact sequence  in ${}^{\perp_{>0}}T$. Then
\begin{itemize}
\item[(a)] We have an exact sequence $0\to Y \to Z\oplus T^{\prime}\to \Omega_{T}^{-}(X)\to0$, where $T^{\prime}\in\add T$.
\item[(b)] For any $A$-module $N$, we have the following long exact sequence 
\begin{align*}
\cdots \to\overline{\Hom}_{A}^{T}(\Omega_{T}^{-n}(Z),N) \to \overline{\Hom}_{A}^{T}(\Omega_{T}^{-n}(Y),N) \to \overline{\Hom}_{A}^{T}(\Omega_{T}^{-n}(X),N) \to \cdots \\
\cdots \to\overline{\Hom}_{A}^{T}(\Omega_{T}^{-}(X),N) \to\overline{\Hom}_{A}^{T}(Z,N) \to \overline{\Hom}_{A}^{T}(Y,N) \to \overline{\Hom}_{A}^{T}(X,N).
\end{align*}
\end{itemize}
\end{proposition}
\begin{proof}
(a) Let $h : X\to T^{\prime}$ be an injective left $(\add T)$-approximation of $X$.
Since $\Ext_{A}^{1}(Z,T)=0$, $h$ factors through $f$, and therefore we have the following commutative diagram
\[\xymatrix{
0 \ar[r] &X \ar[r]\ar@{=}[d] & Y \ar[r]\ar[d] & Z \ar[r]\ar[d] & 0\\
0 \ar[r] &X \ar[r] &T^{\prime} \ar[r] & \Omega_{T}^{-}(X) \ar[r] & 0.
}\]
Thus we have an exact sequence $0\to Y \to Z\oplus T^{\prime}\to \Omega_{T}^{-}(X)\to0$.

(b) By applying (a) and Lemma \ref{exact} inductively, we have the assertion.
\end{proof}
In the following two propositions, we assume that the global dimension of $A$ is at most $n$.
\begin{proposition}\label{coraddt}
Let $X\in{}^{\perp_{>0}}T$.
If the global dimension of $A$ is at most $n$, then we have $\Omega_{T}^{-n}(X)\in\add T$.
\end{proposition}
\begin{proof}
By Proposition \ref{cotilt} (b), it is enough to show that $\Ext_{A}^{1}(Y,\Omega_{T}^{-n}(X))=0$ for any $Y\in{}^{\perp_{>0}}T$.
Let $Y\in{}^{\perp_{>0}}T$.
By using Proposition \ref{cotilt} (a), we have the following exact sequence
\[
0 \to X \to T_{0} \xto{f_{0}} T_{1} \xto{f_{1}} \cdots \to T_{n-1} \xto{f_{n-1}}  \Omega_{T}^{-n}(X) \to 0,
\]
where $T_{i}\in\add T$ and $\Im f_{i}=\Omega_{T}^{-(i+1)}(X)$.
By applying $\Hom_{A}(Y,-)$ to this exact sequence, we have the following isomorphisms
	\begin{align*}
	\Ext^{1}_{A}(Y,\Omega_{T}^{-n}(X)) & \simeq \Ext_{A}^{2}(Y,\Omega_{T}^{-(n-1)}(X))\\
	& \simeq \Ext_{A}^{3}(Y,\Omega_{T}^{-(n-2)}(X))\\
	& \cdots \\
	& \simeq \Ext_{A}^{n+1}(Y,X)=0,
	\end{align*}
where the last equation follows from $\gl A\leq n$.
\end{proof}
\begin{proposition}\label{corexact}
Assume that the global dimension of $A$ is at most n.
For an exact sequence $0\to X \to Y \to Z \to 0$ in ${}^{\perp_{>0}}T$ and any $A$-module $N$, we have the following exact sequence
\begin{align*}
0 \to\overline{\Hom}_{A}^{T}(\Omega_{T}^{-(n-1)}(Z),N) \to \overline{\Hom}_{A}^{T}(\Omega_{T}^{-(n-1)}(Y),N) \to \overline{\Hom}_{A}^{T}(\Omega_{T}^{-(n-1)}(X),N) \to \cdots \\
\cdots \to\overline{\Hom}_{A}^{T}(\Omega_{T}^{-}(X),N) \to\overline{\Hom}_{A}^{T}(Z,N) \to \overline{\Hom}_{A}^{T}(Y,N) \to \overline{\Hom}_{A}^{T}(X,N).
\end{align*}
\end{proposition}
\begin{proof}
By Proposition \ref{coraddt}, we have $\overline{\Hom}_{A}^{T}(\Omega_{T}^{-n}(X),N)=0$.
Therefore, we have a desired exact sequence by Proposition \ref{longexa} (b).
\end{proof}
Then we prove Theorem \ref{end2}.
\begin{proof}[Proof of Theorem \ref{end2}]
We show that projective dimensions of all right $\End_{A}(M)/[T]$-modules are at most $3n-1$.
Let $N$ be a right $\End_{A}(M)/[T]$-module.
There exist $X,Y\in{}^{\perp_{>0}}T$ and a homomorphism of $A$-modules $f:X\to Y$ which induce a minimal projective presentation of $N$,
\[
\overline{\Hom}_{A}^{T}(Y,M) \xto{-\circ f} \overline{\Hom}_{A}^{T}(X,M) \to N \to 0.
\]
Let $g:X\to T^{\prime}$ be an injective left $(\add T)$-approximation of $X$.
We have an injective morphism $h=f\oplus g: X \to Y\oplus T^{\prime}$.
Since $g$ is a left $(\add T)$-approximation of $X$, we have $\Cok h\in{}^{\perp_{>0}}T$.
Let $Z=\Cok h$.
We have an exact sequence $0\to X\xto{h}Y\oplus T^{\prime}\to Z\to0$.
By Proposition \ref{corexact}, we have the following exact sequence
\begin{align*}
0 \to\overline{\Hom}_{A}^{T}(\Omega_{T}^{-(n-1)}(Z),N) \to \overline{\Hom}_{A}^{T}(\Omega_{T}^{-(n-1)}(Y),N) \to \overline{\Hom}_{A}^{T}(\Omega_{T}^{-(n-1)}(X),N) \to \cdots \\
\cdots \to\overline{\Hom}_{A}^{T}(\Omega_{T}^{-}(X),N) \to\overline{\Hom}_{A}^{T}(Z,N) \to \overline{\Hom}_{A}^{T}(Y,N) \xto{-\circ f} \overline{\Hom}_{A}^{T}(X,N).
\end{align*}
Therefore the projective dimension of $N$ is at most $3n-1$.
\end{proof}
\section*{Acknowledgements}
The author is supported by Grant-in-Aid for JSPS Fellowships 15J02465.
He would like to thank my supervisor Osamu Iyama for his support and many helpful comments.
He is grateful to Kota Yamaura for helpful comments and discussions.
The author thanks Yuya Mizuno, Gustavo Jasso, and Takahide Adachi for taking care of me.

\end{document}